\documentclass[alpha-refs]{wiley-article}

\usepackage{siunitx}
\usepackage[T1]{fontenc}
\usepackage{nameref}
\usepackage{amsfonts, amsmath, amssymb, thmtools}
\usepackage[colorlinks=true, urlcolor=blue, citecolor=blue, linkcolor=blue]{hyperref}
\usepackage{cleveref}

\declaretheorem[name=Assumption, refname={Assumption, Assumptions}, Refname={Assumption, Assumptions}, sibling=theorem]{assumption}
\declaretheorem[name=Notation, refname={Notation, Notations}, Refname={Notation, Notations}, sibling=theorem]{notation}

\setlist[enumerate]{leftmargin=*}
\usepackage{tikz}
\usepackage{float, epsfig, subcaption}
\usepackage{graphicx}
\graphicspath{{figures/}{../figures/}}
\def\I{{\mathbb I}}
\def\I{{\mathbb I}}
\def\P{{\mathbb P}}
\def\E{{\mathbb E}}
\def\V{{\mathbb V}}
\def\X{{\vec{X}}}
\def\x{{\vec{x}}}

\def\y{{\vec{y}}}
\def\Y{{\vec{Y}}}
\def\Re{{\mathbb R}}

\def\half{\frac{1}{2}}

\def\cc{{\vec{c}}}
\def\I{{\mathbb I}}
\def\subH1{{ H_1'}}

\def\x{{\vec{x}}}
\def\mX{{\mathbb{X}}}
\def\X{{\vec{X}}}
\def\Y{{\vec{Y}}}
\def\y{{\vec{y}}}

\def\E{{\mathbb E}}
\def\P{{\mathbb P}}
\def\V{{\mathbb V}}
\def\Re{{\mathbb R}}
\def\sD{\mathcal{D}}
\def\sH{\mathcal{H}}
\def\sP{\mathcal{P}}
\def\sX{\mathcal{X}}

\def\to{{\theta_1}}
\def\ts{{\theta^*}}
\def\tz{{\theta_0}}
\def\Tz{{\Theta_0}}
\def\Tzc{{\Theta_0^c}}

\def\limn{\lim_{n \rightarrow \infty}}

\def\Beps{{\boldsymbol{\epsilon}}}
\def\Bbeta{{\boldsymbol{\beta}}}
\def\Bgamma{{\boldsymbol{\gamma}}}

\renewcommand{\vec}[1]{\mathbf{#1}}

\newcommand{\seqn}[1]{(#1_n)_{n \in \mathbb{N}}}

\papertype{}
\paperfield{}

\title{Agnostic tests can control 
the type I and type II errors simultaneously}

\abbrevs{}

\author[1\authfn{1}]{Victor Coscrato}
\author[1\authfn{1}]{Rafael Izbicki}
\author[1\authfn{1}]{Rafael Bassi Stern}

\contrib[\authfn{1}]{Equally contributing authors.}

\affil[1]{Departament of Statistics, Federal University of S\~{a}o Carlos, S\~{a}o Carlos, S\~{a}o Paulo, 13565-905, Brazil}

\corraddress{Rafael Izbicki, Departament of Statistics, Federal University of S\~{a}o Carlos, S\~{a}o Carlos, S\~{a}o Paulo, 13565-905, Brazil}
\corremail{rafaelizbicki@gmail.com}

\fundinginfo{This work was partially supported by \emph{Funda\c{c}\~ao de Amparo \`a Pesquisa do Estado de S\~ao Paulo} (2017/03363-8)}

\runningauthor{Victor Coscrato et al.}

\newcommand{\gfbstfig}{
 \begin{tikzpicture}[thick,scale=0.75, every node/.style={scale=0.9}]
  \draw (-6,2) circle (1.5);
  \draw [fill=lightgray] (-6,2) circle (0.8);
  \draw (-8,0) -- (-4,0) -- (-4,4) -- (-8,4) -- (-8,0);;
  \node at (-6,2) {{\large$R(x)$}}; 
  \node at (-6,0.8) {{\large$H_0$}};
  \node at (-7.5,.50) {{\large$H_0^c$}};
  \node at (-6,4.25) {{\large 	$\phi(x)=0$}};
			
  \draw (-1.4,2.4) circle (1.1);
  \draw [fill=lightgray] (0.2,0.9) circle (0.7);
  \draw (-3,0) -- (1,0) -- (1,4) -- (-3,4) -- (-3,0);;
  \node at (0.2,0.9) {{\large$R(x)$}}; 
  \node at (-1.4,2.4) {{\large$H_0$}};
  \node at (-2.5,.50) {{\large$H_0^c$}}; 
  \node at (-1,4.25) {{\large $\phi(x)=1$}};	
			
  \draw [fill=lightgray] (5.2,1.9) circle (0.7); 
  \draw (4.4,2.4) circle (1.1);
  \draw (2,0) -- (6,0) -- (6,4) -- (2,4) -- (2,0);;
  \node at (5.2,1.9) {{\large$R(x)$}}; 
  \node at (4.2,2.5) {{\large$H_0$}};
  \node at (2.5,.50) {{\large$H_0^c$}};
  \node at (4,4.25) {{\large $\phi(x)=1/2	$}};	
 \end{tikzpicture} \\ 
}

\begin{document}

\maketitle 

\begin{abstract}
Despite its common practice, 
statistical hypothesis testing presents
challenges in interpretation.
For instance, in the standard frequentist framework
there is no control of the type II error.
As a result, the non-rejection of
the null hypothesis $(H_0)$
cannot reasonably be interpreted as its acceptance.
We propose that this dilemma can
be overcome by using agnostic hypothesis tests,
since they can control 
the type I and II errors simultaneously.
In order to make this idea operational,
we show how to obtain agnostic hypothesis
in typical models. For instance, we show how to build
(unbiased) uniformly most powerful agnostic tests and
how to obtain agnostic tests from standard p-values.
Also, we present conditions such that 
the above tests can be made logically coherent.
Finally, we present examples of
consistent agnostic hypothesis tests.

\end{abstract}

\section{Introduction}

Despite its common practice, 
statistical hypothesis testing presents
challenges in interpretation.
For instance, some understand that
an hypothesis test can either accept or reject 
the null hypothesis, $H_0$.
However, in this paradigm 
the probability of
accepting $H_0$ can be 
high even when $H_0$ is false.
Therefore, it is possible to
obtain the undesirable result of
accepting $H_0$ even when 
this hypothesis is unlikely.

In order to deal with this problem,
others propose that an hypothesis test should
either reject or \emph{fail to reject} $H_0$
(\citealt[p. 374]{Casella2002}
and \citealt[p. 545]{Degroot2002}).
Such a position can also lead to
challenges in interpretation,
since the practitioner often wishes
to be able to assert $H_0$ \citep{Levine2008b}.
For example, in regression analysis 
non-significant predictors are
often considered to not affect 
the response variable and are
removed from the model. More generally,
scientists often wish to assert a theory
\citep{Stern2011,Stern2017}.

\citet{Neyman1976}[p.14] briefly introduces 
an alternative to the above paradigms
to hypothesis testing.
In this setting, an hypothesis test 
can have three outcomes:
reject $H_0$, accept $H_0$, or 
remain in doubt about $H_0$ ---
the agnostic  decision.
This third decision allows the hypothesis test to 
commit a less severe error (remain in doubt) whenever
the data doesn't provide strong evidence either
in favor or against the null hypothesis.
This approach, which was called
agnostic hypothesis testing,
was further developed in
\citet{Berg2004,Esteves2016,Esteves2017}.
This framework allows the acceptance of $H_0$
while simultaneously controlling the
type I and II errors through the agnostic decision.
As a result, it is possible to control
the probability that $H_0$ is accepted when
$H_0$ is false.

Although agnostic decisions have been used in classification problems with great success \citep{Lei2014,Jeske2017two,Jeske2017,Sadinle2017}
the agnostic hypothesis testing framework has
only started to be explored.
Here, we generalize to arbitrary hypotheses
the setting in \citet{Berg2004}, 
which applies only to
hypotheses of the form:
$H_i:\theta=\theta_i$, for $i \in \{0,1\}$.
This generalization allows the translation of
standard concepts, such as level, size, power, p-value,
unbiased tests,
and uniformly most powerful test into 
the framework of agnostic hypothesis testing.
Within this  framework, 
we create new
versions of  
standard statistical techniques, such as
t-tests, regression analysis 
and analysis of variance, 
which simultaneously control
type I and type II errors.

\Cref{sec:power} formally defines agnostic tests and
concepts that are used for controlling their error,
such as level, size and power.
\Cref{subsec:ump,subsec:umpu} use these definitions
to generalize the framework in \citet{Berg2004};
they derive agnostic tests that are
uniformly most powerful tests and
unbiased uniformly most powerful tests.
Since it can be hard to obtain the above tests in
complex models, \Cref{sec:pvalue} derives a
general approach for controlling the error of
agnostic tests that is based on p-values.
\Cref{sec:region} advances results that
were obtained in 
\citet{Esteves2016,Esteves2017} and
shows that agnostic tests can
control type I and II errors while 
retaining logical coherence.
\Cref{sec::alternative} discusses how to control
the type I and II errors while 
obtaining consistent agnostic tests.
All proofs are presented in
the supplementary material.

\section{The power of agnostic tests}
\label{sec:power}

We consider a setting in which 
the hypotheses that are tested are 
propositions about a parameter, $\theta$,
that assumes values in the 
parameter space, $\Theta$.
Specifically, the null hypotheses, $H_0$, are 
of the form, $H_0: \theta \in \Tz$,
where $\Tz \subset \Theta$.
The alternative hypotheses, $H_1$, are 
of the form $H_1: \theta \in \Tzc$.
In order to test $H_0$,
we use data, $\X$,
which assumes values on 
the sample space, $\sX$.
Also, $\P_{\tz}$ denotes
the probability measure over $\sX$
when $\theta=\tz \in \Theta$.

$H_0$ is tested through an agnostic test. 
An agnostic test is a function that,
for each observable data point,
determines whether $H_0$ should be
rejected, accepted or remain undecided.
Let 
$\mathcal{D} = \left\{0,\half,1\right\}$ denote
the set of possible outcomes of the test:
accept $H_0$ (0), reject $H_0$ (1), and
remain agnostic $\left(\half\right)$.

\begin{definition}
 \label{def:agnostic}
 An agnostic test is a function,
 $\phi: \sX \rightarrow \sD$.
\end{definition}

\begin{definition}
 \label{def:std}
 An agnostic test, $\phi$, is a
 standard test if Im$[\phi]=\{0,1\}$.
\end{definition}

An agnostic test can have $3$ types of errors.
The type I and type II errors of
agnostic tests are defined in the same way
as those of standard tests. That is, 
a type I error occurs when the test 
rejects $H_0$ and $H_0$ is true.
Similarly, a type II error occurs when
the test accepts $H_0$ and $H_0$ is false.
A type III error occurs whenever the
test remains agnostic.
An agnostic test can be designed to
control the errors of type I and II.

\begin{definition}
 An agnostic test, $\phi$, has
 $(\alpha,\beta)$-level if the
 test's probabilities of committing errors
 of type I and II are controlled
 by, respectively, $\alpha$ and $\beta$.
 That is,
 \begin{align*}
  \alpha_{\phi} &:= 
  \sup_{\tz \in H_0}
  \P_{\tz}(\phi=1) = \alpha \\
  \beta_{\phi} &:= 
  \sup_{\to \in H_1}
  \P_{\to}(\phi=0) = \beta
 \end{align*}
 Similarly, $\phi$ has size
 $(\alpha,\beta)$ if the probabilities
 of committing errors of type I and II
 are upper bounded by $\alpha$ and $\beta$.
 That is,
 $\alpha_{\phi} \leq \alpha$ and
 $\beta_{\phi} \leq \beta$.
\end{definition}

Agnostic tests can be compared
by means of their power. The power function of 
a  test is the probability that 
it doesn't commit an error.
That is, the probability that 
it accepts $H_0$ when $H_0$ is true or 
rejects $H_0$ when $H_0$ is false.

\begin{definition}
 The power function of an agnostic test, $\phi$,
 is denoted by $\pi_{\phi}(\theta)$.
 \begin{align*}
  \pi_{\phi}(\theta) &=
  \begin{cases}
   \P_{\theta}(\phi = 0), 
   & \text{if } \theta \in H_0 \\
   \P_{\theta}(\phi = 1),
   & \text{if } \theta \in H_1
  \end{cases}
 \end{align*}
\end{definition}

\begin{definition}
 Let $\phi_1$ and $\phi_2$ be agnostic tests.
 We say that $\phi_1$ is
 uniformly more powerful than $\phi_2$ for
 $H_0$ and write $\phi_1 \succeq \phi_2$ if,
 for every $\theta \in \Theta$,
 $\pi_{\phi_1}(\theta) \geq \pi_{\phi_2}(\theta)$.
\end{definition}

\subsection{Uniformly most powerful tests}
\label{subsec:ump}

\begin{definition}
 \label{def:ump}
 An $(\alpha,\beta)$-level agnostic test, 
 $\phi^*$, is uniformly most powerful (UMP) if,
 for every other $(\alpha,\beta)$-size
 agnostic test, $\phi$, 
 $\phi^* \succeq \phi$.
\end{definition}

In the following,
\cref{assumption:k-r} presents 
general conditions under which 
we can find UMP agnostic tests.
These conditions are the same as
the ones that are typically used
in the  standard frequentist framework
\citep{Casella2002}[p.391].

\begin{assumption} \
 \label{assumption:k-r}
 \begin{enumerate}
  \item For every $\theta \in \Theta$,
  $\P_{\theta}$ is absolutely continuous
  with respect to the Lebesgue measure, 
  $\lambda$, and
  $f_{\theta}(x) := \frac{d\P_{\theta}}{d\lambda}(x) > 0$.
  \item There exists a 
 sufficient statistic for $\theta$, $T$,
 and the likelihood is monotone over $T$.
 \end{enumerate}
\end{assumption}

\Cref{rule::agnostic} and \Cref{thm:k-r} 
present the agnostic tests that 
are UMP under \Cref{assumption:k-r}.

\begin{definition}
 \label{rule::agnostic}
 Let $T$ be a statistic and $c_0 \leq c_1$.
 The agnostic test, $\phi_{T,c_0,c_1}$, is
 \begin{align*}
  \phi_{T,c_0,c_1}(x) &=
  \begin{cases}
   0     & \text{, if $T(x) \leq c_0$} \\
   1     & \text{, if $T(x) > c_1$} \\
   \half & \text{, otherwise.}
   \end{cases}
 \end{align*}
\end{definition}

\begin{theorem}
 \label{thm:k-r}
 Let $H_{0} = \{\theta \in \Theta: \theta \leq \ts\}$, $c_0 \in \Re$ be such that
 $\sup_{\to \in H_1}\P_{\to}(T(X) \leq c_0)=\beta$, and $c_1 \in \Re$ be such that
 $\sup_{\tz \in H_0}\P_{\tz}(T(X) > c_1)=\alpha$.
 Under \Cref{assumption:k-r},
 \begin{enumerate}
  \item If $c_0 \leq c_1$, then
  $\phi_{T,c_0,c_1}$ is an UMP
  $(\alpha,\beta)$-size agnostic test.
  \item If $\alpha$ and $\beta$ are such that $c_0>c_1$ (and thus $\phi_{T,c_0,c_1}$ is not well defined), then let
  $\Phi=\{\phi_{T,c,c}: c_1 \leq c \leq c_0\}$.
  For every $(\alpha,\beta)$-size
  agnostic test, $\phi$,
  there exists $\phi^* \in \Phi$ such that
  $\phi^* \succeq \phi$.   
 \end{enumerate}
\end{theorem}

\Cref{thm:k-r} generalizes
several previous results in the literature.
For example, if 
$\Theta=\{\tz,\to\}$ and 
$T(x)=\frac{f_{\to}(x)}{f_{\tz}(x)}$,
then the likelihood is monotone over $T$.
In this setting, \citet{Berg2004} shows that,
if $c_0 \leq c_1$, then 
$\phi_{T,c_0,c_1}$ is the UMP agnostic test.
Also, one can emulate the 
standard frequentist framework by
not controlling the type II error, that is,
by considering $(\alpha,1)$-size tests.
In this case,
$\Phi=\{\phi_{T,c,c}: c \leq c_0\}$ is
the set of $\alpha$-size
UMP tests in the
standard frequentist framework
\citep{Casella2002}[p.391]. 

Similarly to this case in which $\beta=1$,
the second condition in 
\Cref{thm:k-r} occurs whenever
the control over $\alpha$ and $\beta$ 
is sufficiently weak so that there exist
standard tests of size $(\alpha,\beta)$ and
there is no need of using the agnostic decision.
In this case, the tests in $\Phi$ cannot be
uniformly more powerful than one another
because of a trade-off in the power in
each region of $\Theta$.
If $c_2 < c_3$,
$\phi_2=\phi_{T,c_2,c_2}$ and
$\phi_3=\phi_{T,c_3,c_3}$, then
the comparison of the critical regions of
$\phi_2$ and $\phi_3$ reveals that
the power of $\phi_2$ is higher over $H_1$ and
the power of $\phi_3$ is hgiher over $H_0$.
That is, the choice between the elements in
$\Phi$ depends on the desired balance 
between the power over $H_0$ and over $H_1$.

In the following, \Cref{ex:z-test} presents
an application of \Cref{thm:k-r}.

\begin{example}[Agnostic z-test]
 \label{ex:z-test}
 Let $X_1,\ldots,X_n$ be an i.i.d. sample with 
 $X_i \sim N(\mu,\sigma^2)$, where
 $\mu \in \Re:=\Theta$ and 
 $\sigma^2$ is known.
 Let $H_0=\{\mu \in \Theta: \mu \leq \mu_0\}$
 and $T=\bar{X}$ be the sample mean.
 Note that the conditions in
 \Cref{assumption:k-r} are satisfied.
 Furthermore, if $\alpha+\beta \leq 1$, then
 by taking
 $c_0=\mu_0-\sigma n^{-0.5}\Phi^{-1}(1-\beta)$ and 
 $c_1=\mu_0-\sigma n^{-0.5}\Phi^{-1}(\alpha)$,
 one obtains that $c_0 \leq c_1$,
 $\sup_{\theta \in H_0}\P_{\theta}(T > c_1)=\alpha$ 
 and 
 $\sup_{\theta \in H_1}\P_{\theta}(T \leq c_0)=\beta$.
 Therefore, it follows from 
 \Cref{thm:k-r} that $\phi_{T,c_0,c_1}$ is an
 UMP $(\alpha,\beta)$-level agnostic test.

 \Cref{img::one_sample} illustrates
 the probability of each decision of this test
 as well as its power function
 when $\sigma=1$, $n=10$ and
 $\alpha=\beta=0.05$.

 \begin{figure}[!htpb]
  \centering
  \begin{subfigure}{.5\textwidth}
   \centering
   \includegraphics[width=\linewidth]{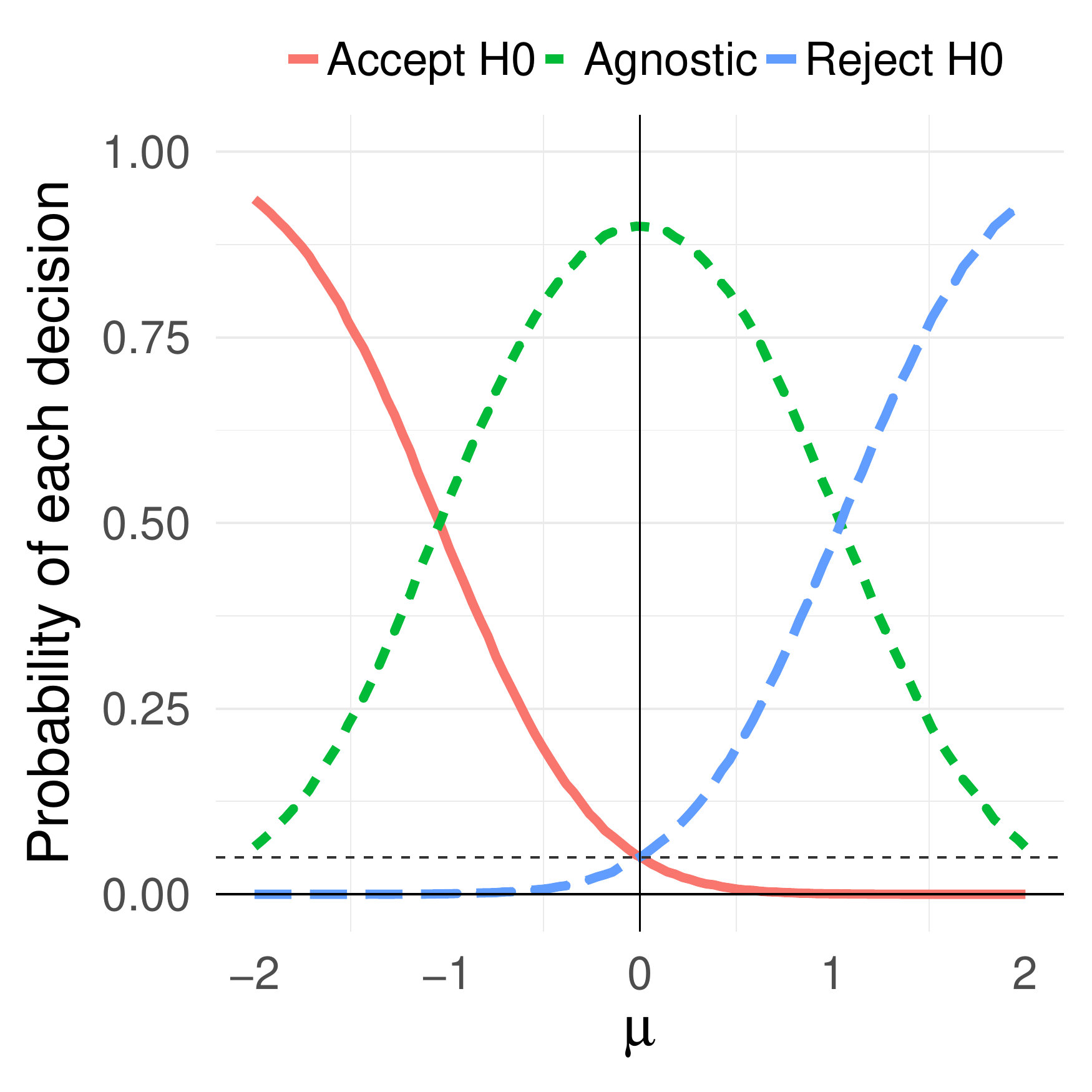}
  \end{subfigure}%
  \begin{subfigure}{.5\textwidth}
   \centering
   \includegraphics[width=\linewidth,page=2]{one_sample}
  \end{subfigure} 
  \caption{Probability of each decision for the UMP $(0.05,0.05)$-level agnostic test for $H_0:\mu \leq  0$ (left) and power function for this test (right). The gray dashed horizontal line shows the values $\alpha=\beta=0.05$.}
  \label{img::one_sample}
 \end{figure}
\end{example}

\subsection{Unbiased uniformly most powerful}
\label{subsec:umpu}

Besides the case studied in \Cref{assumption:k-r},
there often do not exist UMP tests.
For example, they might not exist
when the model has nuisance parameters.
This often occurs because it is possible for
a test to sacrifice power in a region of
$\Theta$ in order to obtain a 
high power in another region.
However, such sacrifices might 
yield undesirable tests.
These tests are characterized 
in the following passage.

An example of an undesirable test is
a test that uses no data.
For example, if $\alpha+\beta \leq 1$ and
$U \sim \text{Uniform}(0,1)$,
then $\phi^U := \phi_{U,\beta,1-\alpha}$ is a test
that uses no data and that 
attains level $(\alpha,\beta)$.
Furthermore, for every $\theta_0 \in H_0$,
$\pi_{\phi^U}(\theta_0) = \beta_{\phi^U}$ and 
also for every $\theta_1 \in H_1$,
$\pi_{\phi^U}(\theta_1) = \alpha_{\phi^U}$.
A generalization of this idea is to consider that a
desirable test, $\phi$, should dominate trivial tests
of the same level, that is,
for every $\theta_0 \in H_0$, 
$\pi_{\phi}(\theta_0) \geq \beta_{\phi}$ and
for every $\theta_1 \in H_1$, 
$\pi_{\phi}(\theta_1) \geq \alpha_{\phi}$.
Such tests are usually called unbiased.

\begin{definition}
 An agnostic test, $\phi$, is unbiased if
 \begin{align*}
  \begin{cases}
   \inf_{\theta_0 \in H_0}
   \P_{\theta_0}(\phi = 0)
   =\pi_{\phi}(\theta_0) 
   &\geq \sup_{\theta_1 \in H_1}
   \P_{\theta_1}(\phi = 0)
   = \beta_{\phi} \\
   \inf_{\theta_1 \in H_1}
   \P_{\theta_1}(\phi = 1)
   =\pi_{\phi}(\theta_1) 
   &\geq \sup_{\theta_0 \in H_0}
   \P_{\theta_1}(\phi = 0)
   = \alpha_{\phi}
  \end{cases}
 \end{align*}
 Note that, if $\phi$ is unbiased, then
 $\alpha_{\phi}+\beta_{\phi} \leq 1$.
\end{definition}

Once only unbiased tests are considered,
it is often possible to find an 
uniformly most powerful test.
In the following,
\Cref{assumption:umpu-1,assumption:umpu-2} 
present general conditions under which 
there exist tests that are
uniformly most powerful among the unbiased tests.
These conditions are the same as
the ones that are typically used
in the  standard frequentist framework
\citep{Lehmann2006}[p.151].

\begin{definition}
 An $(\alpha,\beta)$-level test is
 said to be uniformly most powerful among
 unbiased tests (UMPU) if,
 for every unbiased $(\alpha,\beta)$-size test,
 $\phi$, $\phi^* \succeq \phi$.
\end{definition}

\begin{notation}
 Let $\theta \in \Re^n$. 
 The $i$-th element of $\theta$ 
 is denoted by $\theta(i)$. 
 This notation is useful because
 $\theta_i$ is used to denote 
 an element of $H_i$ and not the
 $i$-th element of $\theta$.
\end{notation}

\begin{assumption} \
 \label{assumption:umpu-1}
 \begin{enumerate}
  \item For every $\theta \in \Theta$,
  $\P_{\theta}$ is absolutely continuous
  with respect to the Lebesgue measure, 
  $\lambda$, and
  $f_{\theta}(x) := \frac{d\P_{\theta}}{d\lambda}(x) > 0$.
  
  \item $\theta \in \Re^n = \Theta$ and
  $f_{\theta}(x)$ is in
  the exponential family, that is,
 there exists $h: \Re \rightarrow \Re^n$
 such that $f_{\theta}(x) = b(x)\exp\left(\theta \cdot h(x) -a(\theta)\right)$.
 
 \item Let $T=(h_{2}(X),\ldots,h_{n}(X))$.
 There exists $V(h(X))$ such that
 $V$ is increasing in $h_{1}(X)$ and
 $T$ and $V$ are independent when
 $\theta(1)=\ts$.
 \end{enumerate}
\end{assumption}

\begin{theorem}
 \label{thm:umpu-1}
 Let $H_{0} = \{\theta \in \Theta: \theta(1) \leq \ts\}$, $\bar{\theta} \in \Theta$ be such that
 $\bar{\theta}(1) = \ts$, 
 $\alpha+\beta \leq 1$, and
 $c_0, c_1 \in \Re$ be such that
 $\P_{\bar{\theta}}(V \leq c_0)=\beta$
 and $\P_{\bar{\theta}}(V > c_1)=\alpha$.
 Under \Cref{assumption:umpu-1},
 $\phi_{V,c_0,c_1}$ is an UMPU
 $(\alpha,\beta)$-level test.
\end{theorem}

\Cref{thm:umpu-1} uses \cref{assumption:umpu-1}
in order to derive UMPU unilateral tests.
Under the stronger conditions in
\cref{assumption:umpu-2} it is also possible
to derive UMPU bilateral tests, 
as presented in \cref{thm:umpu-2}.

\begin{assumption}
 \label{assumption:umpu-2}
 Besides the conditions in \Cref{assumption:umpu-1},
 also include that
 $$V(h_1(x),T) = a(T)h_1(x) + b(t).$$
\end{assumption}

\begin{definition}
 Let $c_{0,l}, c_{1,l}, c_{0,r}, c_{1,r} \in \Re$
 be such that
 $c_{1,l} \leq c_{0,l} \leq c_{0,r} \leq c_{1,r}$.
 \begin{align*}
  \phi_{V,\cc} &=
  \begin{cases}
   1 & \text{, if } 
   V < c_{1,l} \text{ or } V > c_{1,r} \\
   0 & \text{, if } c_{0,l} \leq V \leq c_{0,r} \\
   \half & \text{, otherwise.}
  \end{cases}
 \end{align*}
\end{definition}

\begin{theorem}
 \label{thm:umpu-2}
 Let $H_{0} = \{\theta \in \Theta: \theta(1) = \ts\}$, $\bar{\theta} \in \Theta$ be such that
 $\bar{\theta}(1) = \ts$, 
 $\alpha+\beta \leq 1$,
 and for each $\gamma \in (0,1)$,
 let $c_{\gamma,l}$ and $c_{\gamma,r}$ be such that
 \begin{align*}
   1-\P_{\bar{\theta}}(c_{\gamma,l} \leq V \leq c_{\gamma,r})
   &= \gamma \\
   \E_{\bar{\theta}}[V
   (1-\I(c_{\gamma,l} \leq V \leq c_{\gamma,r}))]
   &= \gamma \E_{\bar{\theta}}[V]
 \end{align*}
 Let 
 $\cc=(c_{1-\beta,l},c_{\alpha,l},c_{\alpha,r},c_{1-\beta,r})$.
 Under \Cref{assumption:umpu-2}
 $\phi_{V,\cc}$ is an UMPU
 $(\alpha,\beta)$-level test.
\end{theorem}

\begin{example}[Agnostic t-test]
 \label{ex:t-test}
 Let $X_1,\ldots,X_n$ be an i.i.d. sample with 
 $X_i \sim N(\mu,\sigma^2)$, where
 $\theta = (\mu,\sigma^2)$ and
 $\Theta = {\rm I\!R}\times{\rm I\!R^+}$.
 Let $H_0^{\leq}=\{(\mu,\sigma^2) \in \Theta: \mu \leq \mu_0\}$ and also
 $H_0^{=}=\{(\mu,\sigma^2) \in \Theta: \mu=\mu_0\}$.
 Let $V=\frac{\sqrt{n}(\bar{X}-\mu_0)}{\sqrt{(n-1)^{-1}\sum_{i=1}^{n}{(X_i-\mu_0)^2}}}$.
 It follows from \citet{Lehmann2006}[p.153] that
 $V$ satisfies the conditions in
 \Cref{assumption:umpu-1,assumption:umpu-2} for
 testing $H_0^{\leq}$ and $H_0^{=}$.
 Therefore, if $\alpha+\beta \leq 1$, then 
 it follows from
 \Cref{thm:umpu-1,thm:umpu-2} that
 $\phi_{V,c_0,c_1}$ and 
 $\phi_{V,\cc}$ are the UMPU tests
 for $H_0^{\leq}$ and $H_0^{=}$.
Moreover, by defining 
 $T(X)=\frac{\sqrt{n}(\bar{X}-\mu_0)}{\sqrt{(n-1)^{-1}\sum_{i=1}^{n}{(X_i-\bar{X})^2}}}$,
 it follows from
 \citet{Lehmann2006}[p.155] that
 $\phi_{V,c_0,c_1}$ and 
 $\phi_{V,\cc}$ are such that
 \begin{align*}
  \phi_{V,c_0,c_1}(x) &=
  \begin{cases}
   0 & T(x) \leq t_{n-1}(\beta) \\
   1 & T(x) > t_{n-1}(1-\alpha) \\
   \half & \text{, otherwise.}
  \end{cases} &
  \phi_{V,\cc}(x) &=
  \begin{cases}
   0 & \text{, if }
   |T(x)| \leq t_{n-1}(0.5(1+\beta)) \\
   1 & \text{, if } 
   |T(x)| > t_{n-1}(1-0.5\alpha) \\
   \half & \text{, otherwise.}
  \end{cases}
 \end{align*}
 where $t_{n-1}(p)$ is the $p$-quantile of
 a Student's t-distribution with
 $n-1$ degrees of freedom.
 \Cref{img::t_test} illustrates the
 probability of each decision for
 $\phi_{V,c_0,c_1}$ and $\phi_{V,\cc}$ when
 $\mu_0=0$, $\sigma^2=1$, $n=10$ and
 $\alpha=\beta=0.05$.
 The power of both tests at $\mu_0=0$ is
 $\beta$. Indeed, it follows from
 \cref{assumption:umpu-1} that the power
 of a $(\alpha,\beta)$-size test at 
 the border points of $H_0$ 
 cannot be higher than $\min(\alpha,\beta)$.
 \begin{figure}
  \centering
  \begin{subfigure}{.5\textwidth}
   \centering
   \includegraphics[width=\linewidth]{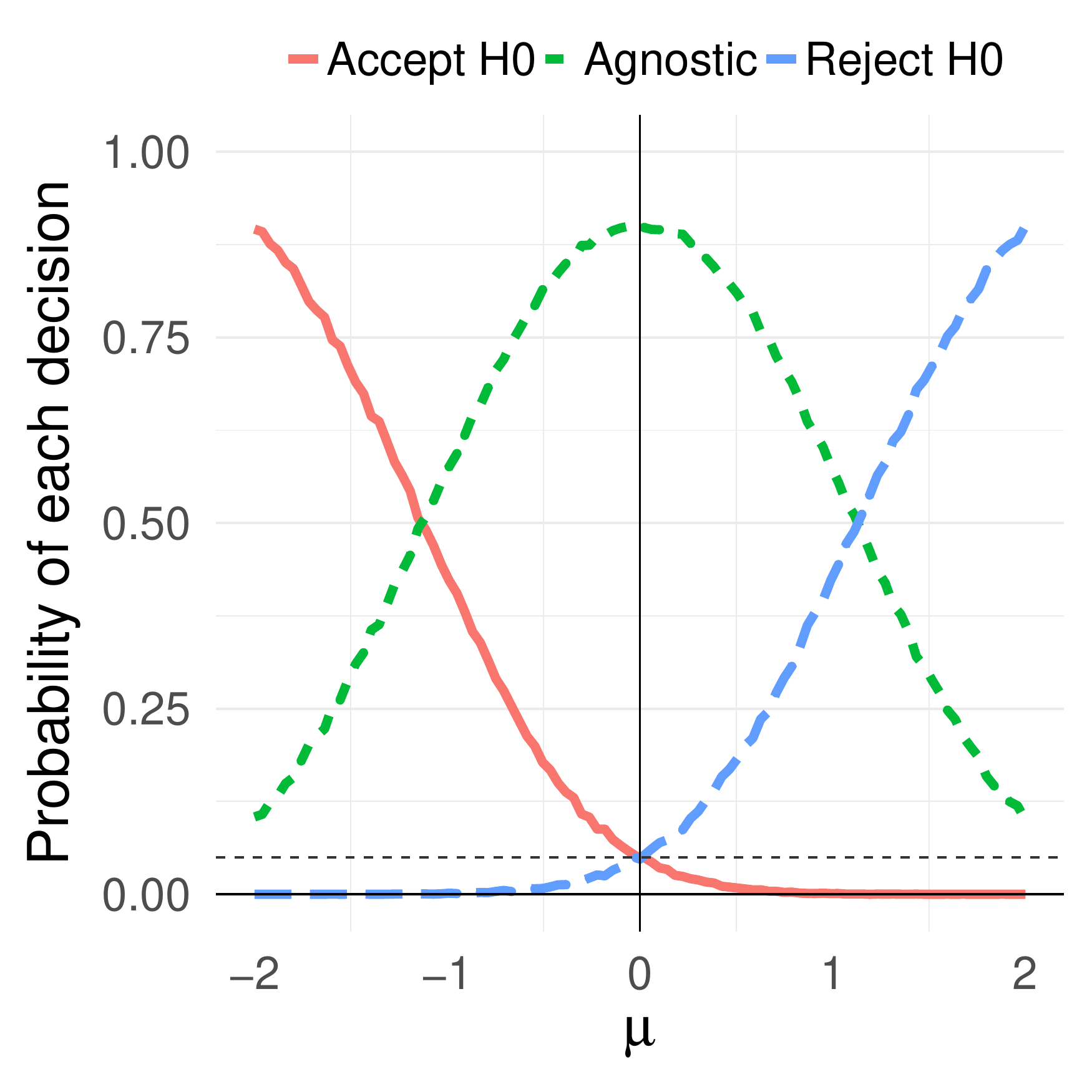}
  \end{subfigure}%
  \begin{subfigure}{.5\textwidth}
   \includegraphics[width=\linewidth]{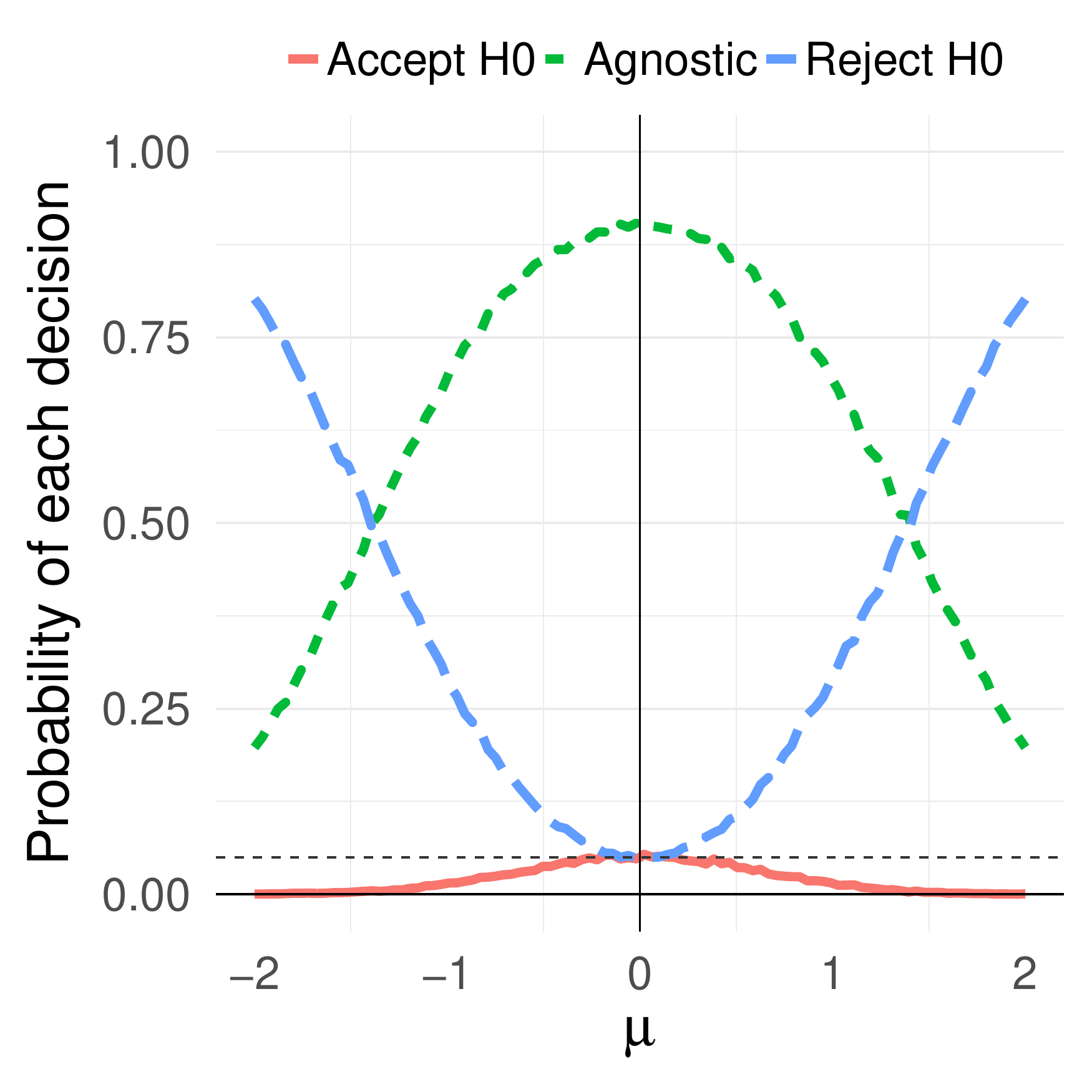}
  \end{subfigure}
  \caption{Probability of each decision 
  for $\phi_{V,c_0,c_1}$ (left) and
  $\phi_{V,\cc}$ (right) when
  $\mu_0=0$, $\sigma^2=1$, $n=10$ and
  $\alpha=\beta=0.05$.}
  \label{img::t_test}
 \end{figure}
\end{example}

\begin{example}[Agnostic linear regression]
 \label{ex:regression}
 Consider a linear regression setting, that is,
 $\Y = \mX\Bbeta +\Beps$, 
 where $d < n$,
 $\Beps \sim N(\textbf{0},\sigma^2 \mathbb{I}_d)$,
 $\mX$ is a $n \times d$ 
 design matrix of rank d and 
 $\Bbeta$ is the $d \times 1$ vector
 with coefficients.
 For a fixed $k \in \Re^{d}$ and $c \in \Re$, let
 $H_0^\leq: k \cdot \Bbeta \leq c$ and
 $H_0^=: k \cdot \Bbeta = c$.
 Let $\alpha+\beta \leq 1$.
 By taking $\hat{\Bbeta}=(\mX^{t}\mX)^{-1}\mX^{t}Y$,
 the least squares estimator for $\Bbeta$,
 it follows from \citet{Shao2003}[p.416] that
 $V=\frac{k^{t}\hat{\beta}-c}{\sqrt{k^{t}(\mX^{t}\mX)^{-1}k}\|Y\|^{2}_{2}(n-d)^{-1}}$
 satisfies the conditions in
 \cref{assumption:umpu-1,assumption:umpu-2}.
 Therefore, the UMPU tests,
 $\phi_{V,c_0,c_1}$ and $\phi_{V,\cc}$, 
 are such that
 \begin{align*}
  \phi_{V,c_0,c_1}(x) &=
  \begin{cases}
   0 & V(x) \leq t_{n-d}(\beta) \\
   1 & V(x) > t_{n-d}(1-\alpha) \\
   \half & \text{, otherwise.}
  \end{cases} &
  \phi_{V,\cc}(\x) &=
  \begin{cases}
   0 & \text{, if }
   |V(x)| \leq t_{n-d}(0.5(1+\beta)) \\
   1 & \text{, if } 
   |V(x)| > t_{n-d}(1-0.5\alpha) \\
   \half & \text{, otherwise.}
  \end{cases}
 \end{align*}
 where $t_{n-d}(q)$ denotes the 
 $q$ quantile of Student's t-distribution
 with $n-d$ degrees of freedom.
\end{example}

\section{General agnostic tests of a given level}
\label{sec:pvalue}

Oftentimes, an UMPU agnostic test 
does not exist or is difficult to derive.
In such a situation, one might be 
willing to use an $(\alpha,\beta)$-level test
that is not uniformly most powerful.
A wide class of such tests can 
be obtained through the 
p-value of standard hypothesis tests.
The definition of p-value is revisited below. 

\begin{definition}
 A nested family of standard tests for $H_0$,
 $\Phi$, is such that
 \begin{enumerate}
  \item For every $\phi \in \Phi$,
  $\phi$ is a standard test.
  \item The function $g: \Phi \rightarrow [0,1]$,
  $g(\phi)=\alpha_{\phi}$ is bijective.
  \item If $\phi_1,\phi_2 \in \Phi$ and
  $\alpha_{\phi_1} \leq \alpha_{\phi_2}$, then
  $\{x \in \sX: \phi_1(x)=1\} \subset \{x \in \sX: \phi_2(x)=1\}$.
 \end{enumerate}
\end{definition}

\begin{example}
 Let $\lambda(x) = -\log\left(\frac{\sup_{\tz \in H_0}f_{\tz}(x)}{\sup_{\theta \in \Theta}f_{\theta}(x)}\right)$. The collection of
 generalized likelihood ratio tests,
 $\Phi=\{\phi_{\lambda,k,k}: k \geq 0\}$, is
 a nested family of standard tests for $H_0$.
\end{example}

\begin{definition}
\label{def::pvalue}
 Let $\Phi$ denote a nested family
 of standard tests for $H_0$.
 The p-value of $\Phi$ against $H_0$,
 $p_{H_0,\Phi}: \sX \rightarrow [0,1]$ is 
 such that
 $p_{H_0,\Phi}(x) = \inf\{\alpha_{\phi}: \phi \in \Phi \wedge \phi(x) = 1\}$.
\end{definition}

Intuitively, if $H_0$ is rejected 
whenever the p-value is smaller than $\alpha$,
then the type I error is controlled by $\alpha$.
Similarly, one might expect that
if $H_0$ is accepted whenever 
the p-value is larger than $1-\beta$, then
the type II error is controlled by $\beta$.
\cref{thm:p} provides conditions under which
this reasoning is valid.

\begin{theorem}
 \label{thm:p}
 Let $\Phi$ be a nested family of
 standard tests for $H_0$ such that,
 for every $\phi \in \Phi$, 
 $\phi$ is an unbiased test.
 Assume that $\Theta$ is
 a connected space and that,
 for every $x \in \sX$,
 $\P_{\theta}(p_{H_0,\Phi}(x) \leq t)$
 is a continuous function over $\theta$. Let
 $p=p_{H_0,\Phi}$.
 Then,  the test
 $\phi_{1-p,\beta,1-\alpha}$, i.e.,
  \begin{align*}
  \phi_{1-p,\beta,1-\alpha}(x) &=
  \begin{cases}
   0     & \text{, if $p(x) \geq 1-\beta$} \\
   1     & \text{, if $p(x) < \alpha$} \\
   \half & \text{, otherwise.}
   \end{cases}
 \end{align*}
  is a $(\alpha,\beta)$-level test for $H_0$.
\end{theorem}

\begin{example}[{General Linear Hypothesis in 
Regression Analysis}]
 \label{ex:gen-regression}
 Consider the linear regression setting 
 (\cref{ex:regression}) and
 the general linear hypothesis
 \begin{align*}
  H_0: \mathbb{K} \Bbeta = \Bgamma_0
 \end{align*}
 where $\mathbb{K}$ is a $q \times d$ matrix
 and $\Bgamma_0 \in \mathbb{R}^q$.
 A particular case of this problem
 is the ANOVA test \citep{Neter1996}.
 There exists no UMPU test for $H_0$  \citep{Geisser2006}. 
 However, the F-statistic
 \begin{align*}
  F &= \frac{(\mathbb{K} \widehat{\Bbeta}-\Bgamma_0)^t (\mathbb{K} (\mathbb{X}^t \mathbb{X})^{-1}\mathbb{K}^t)^{-1} (\mathbb{K} \widehat{\Bbeta}-\Bgamma_0)q^{-1}}{ (\y- \mathbb{X}\Bbeta)^t(\y- \mathbb{X}\Bbeta)(n-p)^{-1}}
 \end{align*}
 is such that, for every $k \geq 0$,
 $\phi_{F,k,k}$ is unbiased for $H_0$
 \citep{Monahan2008}.
 Furthermore, it can be shown that
 $p_{H_0,\Phi}=F_{q,n-1}(F)$, where
 $F_{q,n-1}(\cdot)$ denotes the cumulative distribution function
 of a
 Snedecor's F-distribution
 random variable
 with $(q,n-1)$
 degrees of freedom.
 Since all conditions in 
 \cref{thm:p} are satisfied,
 $\phi_{1-F_{q,n-1}(F),\beta,1-\alpha}$
 is a $(\alpha,\beta)$-level test.
\end{example}

\begin{example}[{Permutation Test}]
 Let $\X=(X_1,\ldots,X_m)$ and
 $\Y=(Y_1,\ldots,Y_n)$ be i.i.d. samples from
 continuous distributions, $F_{X}$ and $F_{Y}$.
 Also, consider that $H_0:F_{X}=F_{Y}$ and 
 $\Theta=\{(F_{X},F_{Y}): F_{X} \mbox{ is stocastically larger than } F_{Y}\}$.
 Let $p_{H_0}(\X,\Y)$ be a p-value based on 
 a permutation test such that, if
 $\Y'=(Y'_1,\ldots,Y'_n)$ is such that,
 for every $i=1,\ldots,n$, $y'_i \geq y_i$,
 then $p_{H_0}(\X,\Y') \geq p_{H_0}(\X,\Y)$.
 It follows from \citet[Lemma 5.9.1]{Lehmann2006}
 that $p_{H_0}$ is unbiased for $H_0$.
 Also, under the topology induced by 
 the total variation metric,
 $\Theta$ is connected and
 $\P_{\theta}(p_{H_0} \leq t)$ is
 continuous over $\theta$.
 Conclude from \cref{thm:p} that
 $\phi_{1-p_{H_0},\beta,1-\alpha}$ is
 a $(\alpha,\beta)$-level agnostic test.
\end{example}

\section{Connections to region estimation}
\label{sec:region}

There exist several known equivalences between
standard tests and region estimators
\citep[p.241]{Bickel2015}.
For example, every region estimator is equivalent to
a collection of bilateral standard tests.
Also, standard tests for more general hypothesis can
be obtained as the indicator that the
hypothesis intercepts a region estimator.
These connections are useful for 
providing a method of obtaining and
interpreting standard hypothesis tests.

The following subsections show that 
similar results hold for 
the agnostic tests that
were obtained previously.
\Cref{sec:region-unilateral} presents
a general method for 
obtaining agnostic tests from
confidence regions.
Furthermore, it shows how this method
relates to logical coherence and to 
the unilateral tests in \cref{sec:power}.
\Cref{sec:region-bilateral} presents 
an equivalence equivalence between 
nested region estimators and
collections of bilateral agnostic tests.

\subsection{Agnostic tests based on a region estimator}
\label{sec:region-unilateral}

An agnostic test can
have other desirable properties
besides controlling both the
type I and type II errors.
For instance, \citet{Esteves2016,Esteves2017} 
show that agnostic tests can be 
made logically consistent.
That is, it is possible to test several
hypothesis using agnostic hypothesis tests
in such a way that it is impossible
to obtain  logical contradictions between 
their conclusions. 
This property generally cannot be 
obtained using standard tests \citep{Izbicki2015}.
Logically consistent agnostic tests are
connected to region estimators,
as summarized below.

\begin{definition}
 A region estimator is a function
 $R: \sX \rightarrow \sP(\Theta)$.
\end{definition}

\begin{definition}[Agnostic test based on a region estimator]
\label{def::region}
Let $R(x)$ be a region estimator and
 $H_0 \subseteq \Theta$.
 The agnostic test based on $R$ for
 testing $H_0$, $\phi_{H_0,R}$ is such that
 \begin{align*}
  \phi_{H_0,R}(x) &=
  \begin{cases}
   0 & \text{, if } R(x) \subseteq H_0 \\
   1 & \text{, if } R(x) \subseteq H_0^c \\
   \half & \text{, otherwise.} \\
  \end{cases}
 \end{align*}

 \Cref{fig::region} illustrates this procedure.
 
 \begin{figure}
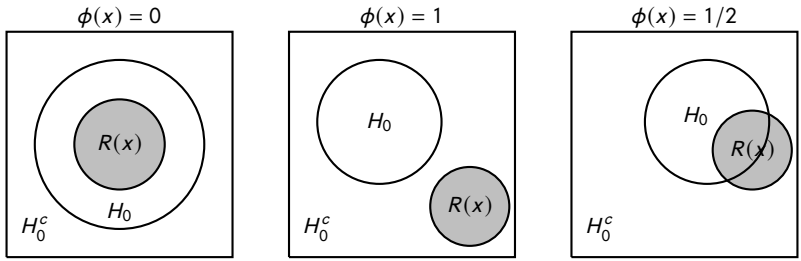

  \center
  \gfbstfig
  \mbox{} \vspace{-2mm} \mbox{} \\ 
  \caption{$\phi(x)$ is
  an agnostic test based on 
  the region estimator, $R(x)$, 
  for testing $H_0$.}
  \label{fig::region}
 \end{figure}
\end{definition}

\begin{definition}
 A collection of tests,
 $(\phi_{H_0})_{H_0 \in \sH}$
 is based on a region estimator if
 there exists a region estimator, $R(x)$,
 such that, for every $H_0 \in \sH$,
 $\phi_{H_0}$ is based on $R$.
\end{definition}

\begin{theorem}[\citet{Esteves2016}]
 \label{thm:region-logical}
 Let $(\phi_{H_0})_{H_0 \in \sigma(\Theta)}$
 be a collection of agnostic tests such that
 $\sigma(\Theta)$ is a $\sigma$-field over $\Theta$
 and, for every $\theta \in \Theta$,
 $\{\theta\} \in \sigma(\Theta)$.
 $(\phi_{H_0})_{H_0 \in \sigma(\Theta)}$ is
 logically consistent if and only if
 it is based on a region estimator.
\end{theorem}

It follows from \cref{thm:region-logical} that
the collection of tests based on 
a region estimator is logically consistent.
\cref{thm:region-control} shows that,
if this region estimator has
confidence $1-\alpha$, then
the tests based on it also control both
the type I and II errors by $\alpha$.

\begin{theorem}
 \label{thm:region-control}
 If $R(x)$ is a region estimator for
 $\theta$ with confidence $1-\alpha$ and
 $\phi_{H_0,R}$ is an agnostic test for $H_0$
 based on $R$, then
 $\phi_{H_0,R}$ is a 
 $(\alpha,\alpha)$-size test.
\end{theorem}

Furthermore, the unilateral tests that
were developed in \Cref{sec:power,sec:pvalue} 
are based on confidence regions.
In order to present such regions, 
\cref{thm:region-unilateral} uses
\cref{assumption:region-unilateral-1,assumption:region-unilateral-2}.

\begin{assumption}
 \label{assumption:region-unilateral-1}
Let $H_{0,\theta^*}: \theta(1) \leq \theta^*$. $\left(\phi_{H_{0,\theta^*}}\right)_{\theta^* \in \Re}$ is a collection of
 agnostic tests such that
 \begin{enumerate}[label=(\alph*)]
  \item If $\to \leq \theta_2$ and
  $\phi_{H_{0,\to}}(x)=0$, then
  $\phi_{H_{0,\theta_2}}(x)=0$
  \item If $\to \leq \theta_2$ and
  $\phi_{H_{0,\theta_2}}(x)=1$, then
  $\phi_{H_{0,\to}}(x)=1$.
 \end{enumerate}
\end{assumption}

\begin{assumption}
 \label{assumption:region-unilateral-2}
 Let $H_{0,\theta^*}: \theta(1) \leq \theta^*$.
 $\left(\phi_{H_{0,\theta^*}}\right)_{\theta^* \in \Re}$ is a collection of
 agnostic tests such that
 for every $\theta \in \Theta$
 such that $\theta(1)=\theta^*$, 
 $\P_{\theta}\left(\phi_{H_{0,\theta^*}}=\half\right)\geq 1-2\alpha$
\end{assumption}

\Cref{assumption:region-unilateral-1} requires that a
collection of unilateral tests satisfy a weak form of
logical coherence. That is, if 
$\to \leq \theta_2$ and 
the collection of tests
accepts that $\theta \leq \to$, then 
it accepts that $\theta \leq \theta_2$.
Similarly, if $\to \leq \theta_2$ and
the collection of tests rejects 
that
$\theta \leq \theta_2$, then 
it also rejects that $\theta \leq \to$.
\Cref{assumption:region-unilateral-2} requires that,
for every test in the collection,
the probability of the no-decision alternative
in the border point
of $H_0$ is at least $1-2\alpha$.
\Cref{thm:region-unilateral} shows that 
a collection of unilateral tests that satisfy
\cref{assumption:region-unilateral-1,assumption:region-unilateral-2} is based on 
a confidence region of confidence $1-2\alpha$.

\begin{theorem}
 \label{thm:region-unilateral}
 For each, $\theta^*$, let
 $H_{0,\theta^*}: \theta(1) \leq \theta^*$.
 If $\left(\phi_{H_{0,\theta^*}}\right)_{\theta^* \in \Re}$ satisfies 
 \cref{assumption:region-unilateral-1}, then
 there exists a region estimator,
 $R(x)$, such that, for every $\theta^*$,
 $\phi_{H_{0,\theta^*}}$ is based on $R(x)$.
 Furthermore, if
 \cref{assumption:region-unilateral-2} holds,
 then $R(x)$ is a confidence region for $\theta$
 with confidence $1-2\alpha$.
\end{theorem}

It is possible to use \cref{thm:region-control,thm:region-unilateral} in order 
to extend a collection of unilateral tests to 
a larger collection of tests. If 
the collection of unilateral tests satisfies
\cref{assumption:region-unilateral-1,assumption:region-unilateral-2}, then it follows from 
\cref{thm:region-unilateral} that 
these tests are based on a region estimator,
$R(X)$, with confidence $1-2\alpha$.
Therefore, it follows from 
\cref{thm:region-control} that,
for every $H_0$ of the type
$\theta(1) \in \Theta_0 \subseteq \Re$,
the test for $H_0$ based on $R(X)$
has size $(2\alpha,2\alpha)$. Furthermore,
it follows from \cref{thm:region-logical} that
the collection of these tests is
logically coherent.
\cref{cor-izbicki} summarizes 
these conclusions.

\begin{corollary}
 \label{cor-izbicki}
 For each, $\theta^*$, let
 $H_0: \theta(1) \leq \theta^*$.
 Also, assume that 
 $\left(\phi_{H_{0,\theta^*}}\right)_{\theta^* \in \Re}$ satisfies \cref{assumption:region-unilateral-1,assumption:region-unilateral-2}.
 Let $R(X)$ be such as in 
 \cref{thm:region-unilateral}.
 Consider the collection of agnostic tests $(\phi_{H_{0,\Tz},R})_{\Tz \subset \Re}$,
 where $H_{0,\Tz}: \theta(1) \in \Tz$ (recall \Cref{def::region}). 
 Then
 \begin{enumerate}[label=(\roman*)]
\item this collection
 is
 logically coherent,
\item each test is this collection   has size
 $(2\alpha,2\alpha)$, and
\item this collection is an extension of the collection $\left(\phi_{H_{0,\theta^*}}\right)_{\theta^* \in \Re}$
\end{enumerate}
\end{corollary}

Under weak conditions, the
tests that were developed in
\cref{thm:k-r,thm:umpu-1} satisfy
\cref{assumption:region-unilateral-1,assumption:region-unilateral-2}.
As a result, they can be used in
\cref{thm:region-unilateral,cor-izbicki}.
These results are presented in
\cref{cor:ump-region,cor:umpu-region} and
illustrated in
\cref{ex:z-test-region,ex:t-test-region}.

\begin{corollary}
 \label{cor:ump-region}
Consider the setting of \cref{thm:k-r}, and let $H_{0,\theta^*}: \theta \leq \theta^*$.
 The collection $\phi_{H_{0,\theta^*}}$ 
 of UMP $(\alpha,\alpha)$-level test presented 
 in \cref{thm:k-r} is based on 
 a region estimator, $R(X)$.
 Furthermore, if $T$ is such that 
 $\P_\theta(T \leq t)$ is
 continuous over $\theta$, then
 $R(X)$ has confidence $1-2\alpha$ for $\theta$.
\end{corollary}

\begin{example}[Agnostic z-test]
 \label{ex:z-test-region} Consider again \cref{ex:z-test}.
 For each $\mu^* \in \Re$,
 let $H_{0,\mu^*}: \mu \leq \mu^*$.
 Let $\alpha \leq 0.5$ and
 $\left(\phi_{H_{0,\mu^*}}\right)_{\mu^* \in \Re}$ be the collection of
 UMP $(\alpha,\alpha)$-level tests in 
 \cref{ex:z-test}. By defining the constants
 $a_1=\sigma n^{-0.5}\Phi^{-1}(1-\alpha)$ and 
 $a_2=\sigma n^{-0.5}\Phi^{-1}(\alpha)$,
 note that
 $\phi_{H_{0,\mu^*}}=\phi_{\bar{X},\mu^*-a_1,\mu^*-a_2}$. It follows that
 $\left(\phi_{H_{0,\mu^*}}\right)_{\mu^* \in \Re}$ is based on the region estimator
 $R(X)=[\bar{X}-a_1,\bar{X}-a_2]$, which
 is a $1-2\alpha$ confidence interval for $\mu$.
\end{example}

\begin{assumption}
 \label{assumption:umpu-region}
 For each $\theta^* \in \Re$, let
 $V_{\theta^*}$ be such as in
 \cref{assumption:umpu-1} when
 $\theta(1)=\theta^*$.
 There exists a function,
 $g(v,\theta)$, which is
 decreasing over $\theta$ and
 such that $g(V_{\theta},\theta)$
 is ancillary.
\end{assumption}

\begin{corollary}
 \label{cor:umpu-region}
 For each $\theta^* \in \mathbb{R}$,
 let $H_{0,\theta^*}: \theta(1) \leq \theta^*$. 
 Under \cref{assumption:umpu-1} and
 $\alpha \leq 0.5$, let
 $\phi_{H_{0,\theta^*}}$ be the
 UMPU $(\alpha,\alpha)$-level test presented 
 in \cref{thm:umpu-1}. 
 Under \cref{assumption:umpu-region},
 the collection
 $\left(\phi_{H_{0,\theta^*}}\right)_{\theta^* \in \Theta}$ is based on a region estimator,
 $R(X)$, which has 
 confidence $1-2\alpha$ for $\theta$.
\end{corollary}

\begin{example}[Agnostic t-test]
 \label{ex:t-test-region}
 Consider again \cref{ex:t-test}.
 For each $\mu^* \in \Re$,
 let $H_{0,\mu^*}: \mu \leq \mu^*$.
 Let $\alpha \leq 0.5$ and
 $\left(\phi_{H_{0,\mu^*}}\right)_{\mu^* \in \Re}$ 
 be the collection of
 UMP $(\alpha,\alpha)$-level tests in 
 \cref{ex:t-test}. By defining
 $S=\sqrt{(n-1)^{-1}\sum_{i=1}^{n}{(X_i-\bar{X})^2}}$,
 $a_1=n^{-0.5} S t_{n-1}^{-1}(1-\alpha)$ and 
 $a_2=n^{-0.5} S t_{n-1}^{-1}(\alpha)$,
 note that
 $\phi_{H_{0,\mu^*}}=\phi_{\bar{X},\mu^*-a_1,\mu^*-a_2}$. It follows that
 $\left(\phi_{H_{0,\mu^*}}\right)_{\mu^* \in \Theta}$ is based on the region estimator
 $R(X)=[\bar{X}-a_1,\bar{X}-a_2]$, which
 is a $1-2\alpha$ confidence interval for $\mu$.
\end{example}

\subsection{Agnostic tests based on nested region estimators}
\label{sec:region-bilateral}

Contrary to the unilateral tests,
the bilateral tests in \cref{sec:power} are
not based on region estimators.
Indeed, while these bilateral tests can 
accept a precise hypothesis,
this feature cannot be obtained in
tests based on region estimators.
However, similarly to the case for standard tests,
there exists an equivalence between
collections of bilateral agnostic tests and
pairs of nested region estimators.
Indeed, it is possible to obtain from one another 
a nested pair of $1-\alpha$ and $\beta$ confidence regions
and a collection of bilateral $(\alpha,\beta)$-size tests.
\Cref{defn:region-aprox} prepares for this equivalence,
which is established in \cref{thm:region-bilateral}.

\begin{definition}[Agnostic test based on 
nested region estimators]
 \label{defn:region-aprox}
 Let $R_1(x)$ and $R_2(x)$ be
 region estimators such that,
 $R_1(x) \subseteq R_2(x)$ and
 $H_0 \subseteq \Theta$.
 The agnostic test based on $R_1$ and $R_2$ for
 testing $H_0$, $\phi_{H_0,R_1,R_2}$, is
 \begin{align*}
  \phi_{H_0,R_1,R_2}(x) &=
  \begin{cases}
   0 & \text{, if } H_0 \subseteq R_1 \\
   1 & \text{, if } R_2 \subseteq H_0^c \\
   \half & \text{, otherwise.} \\
  \end{cases}
 \end{align*}
 \Cref{fig:region-aprox} illustrates
 $\phi_{H_0,R_1,R_2}$ when
 $H_0: \theta = \tz$.
\end{definition}

\begin{figure}
 \centering
 \includegraphics[width=0.7\textwidth,trim={1.2cm 2.4cm 1.1cm 2.2cm},clip]{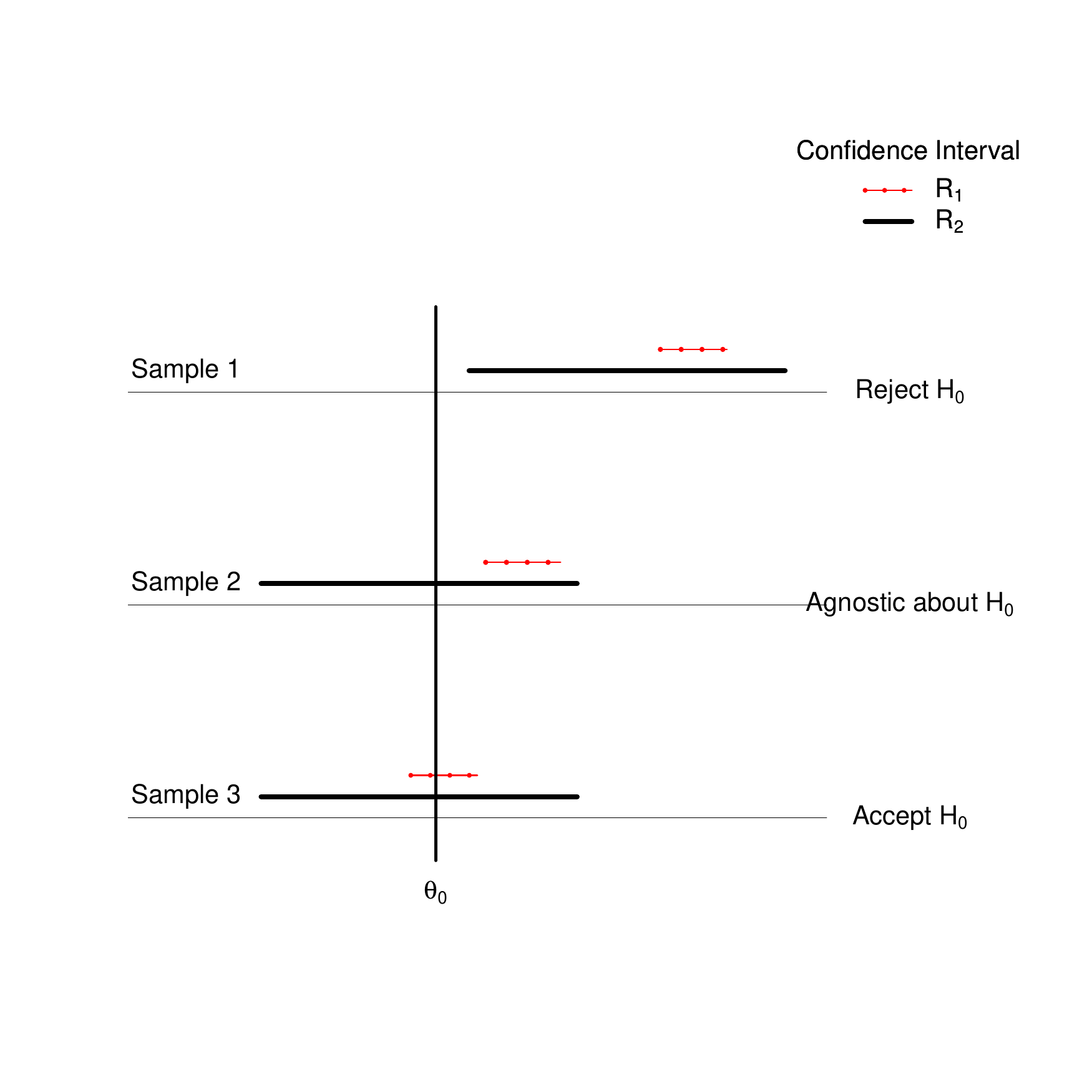}
 \caption{Illustration of the 
 agnostic test based on $R_1$ and $R_2$
 (\cref{defn:region-aprox})
 when $H_0:\theta=\tz$.}
 \label{fig:region-aprox}
\end{figure}

\begin{example}[Agnostic t-test]
 Consider \cref{ex:t-test}.
 For each $\mu^* \in \Re$,
 let $H_{0,\mu^*}: \mu = \mu^*$.
 The UMPU agnostic test is
 based on the region estimators
 \begin{align*}
  R_1(x) &=
  \left[\bar{X}-t_{n-1}(0.5(1+\beta))\sqrt{S^2/n},
  \bar{X}+t_{n-1}(0.5(1+\beta))\sqrt{S^2/n}\right] 
  \text{, and} \\
  R_2(x) &=
  \left[\bar{X}-t_{n-1}(1-0.5\alpha)\sqrt{S^2/n},
  \bar{X}+t_{n-1}(1-0.5\alpha)\sqrt{S^2/n}\right]
 \end{align*}
\end{example}

\begin{theorem}
 \label{thm:region-bilateral}
 For each $\theta^*$, let
 $H_{0,\theta^*}: \theta(1)=\theta^*$.
 \begin{enumerate}
  \item If $R_1(x) \subseteq R_2(x)$ are
  confidence regions for $\theta$  with
  confidence $1-\beta$ and $\alpha$, then
  for every $\theta^* \in \Re$,
  $\phi_{H_{0,\theta^*},R_1,R_2}$ is a
  $(\alpha,\beta)$-size test.
  \item Let 
  $\left(\phi_{H_{0,\theta^*}}\right)_{\theta^* \in \Re}$ be
  a collection of $(\alpha,\beta)$-size tests.
 If for every $\theta \in \Theta$ such that
 $\theta(1)=\theta^*$, 
 $\P_{\theta}(\phi_{H_{0,\theta^*}}=0) = \beta$ and
 $\P_{\theta}(\phi_{H_{0,\theta^*}}=1) = \alpha$, then
 there exist region estimators, 
 $R_1(x)$ and $R_2(x)$, such that
 $R_1(x) \subseteq R_2(x)$, 
 $R_1(x)$ and $R_2(x)$ are
 confidence regions for $\theta$ with, 
 respectively,  confidence $1-\beta$ and $\alpha$ and
 such that $\phi_{H_{0,\theta^*}}$ is 
 based on $R_1(x)$ and $R_2(x)$.
 \end{enumerate}
\end{theorem}

\section{Consistent agnostic tests}
\label{sec::alternative}

A sequence of agnostic tests,
which is indexed on the sample size,
is consistent if there exists a large enough sample
such that, with high probability, 
the test accepts $H_0$ under $H_0$ and
reject $H_1$ under $H_1$.
That is, a sequence of agnostic tests is consistent if 
the respective sequence of power functions
converges to $1$ as the sample size goes to infinity.
This notion is formalized in \cref{def:consistent}.

\begin{definition}
 \label{def:consistent}
 A sequence of agnostic tests for $H_0$,
 $\seqn{\phi}$, is
 consistent if, 
 for every $\theta \in \Theta$,
 $\limn \pi_{\phi_n}(\theta)=1$.
\end{definition}

Under a wide variety of models,
it is impossible to obtain 
consistent agnostic tests.
A class of such models is described in
\cref{assumption:connected}.

\begin{assumption}[Non-separability between $H_0$ and $H_1$]
 \label{assumption:connected}
 \begin{enumerate} \
  \item $\Theta$ is connected.
  \item $H_0 \notin \{\emptyset, \Theta\}$.
  \item $\seqn{\phi}$ is a sequence of
  agnostic tests for $H_0$ such that,
  for every $n \in \mathbb{N}$ and
  $i \in \left\{0,\half,1\right\}$,
  $\P_{\theta}(\phi_n=i)$ is 
  continuous over $\theta$.
 \end{enumerate}
\end{assumption}
 
\Cref{assumption:connected} is met in
the examples presented in
\cref{sec:power,sec:pvalue}.
\Cref{thm:no-consistency} shows that,
under \cref{assumption:connected},
it is impossible to obtain a 
consistent sequence of hypothesis test.

\begin{theorem}
 \label{thm:no-consistency}
 Under \cref{assumption:connected},
 if $\seqn{\phi}$ is a sequence of
 $(\alpha, \beta)-size$ tests,
 where $\max(\alpha, \beta) < 1$, then
 $\seqn{\phi}$ is not consistent.
 Furthermore, under the same assumption, 
 if $\limn \alpha_n = 0$, 
 $\limn \beta_n = 0$ and
 $\seqn{\phi}$ is a sequence of
 $(\alpha_n,\beta_n)$-size tests,
 then for some $ \theta \in \Theta$,
 $\limn \P_{\theta}\left(\phi_n=\half\right)=1$.
\end{theorem}

Despite \cref{thm:no-consistency},
consistency can be obtained by relaxing
the control over the test's errors.
In particular, one might drop
the requirement that the type II error probability
be controlled uniformly over 
all points in the alternative hypothesis.
The remainder of this section explores
alternative methods of 
controlling the type II error probabilities.

One alternative way to 
control the type II error probabilities is 
to require solely that
$\sup_{\theta \in \subH1} \P_{\theta}(\phi=0) \leq \beta$,
where $\subH1$ is a subset of $H_1$ which 
is relevant for the practitioner.
One procedure to choose $\subH1$ in practice
is to determine a desired effect size
through expert knowledge elicitation.
The effect size is often easier to interpret than
the value of the parameter itself.
This procedure is similar to what is 
often done in power calculations \citep{Neter1996}.

\begin{example}\emph{(Agnostic linear regression)}
 \label{ex:linearEffect}
 Consider the linear regression setting in 
 \cref{ex:regression}. Also, one wishes to
 test the hypothesis $H_0: \beta_k=0$ with 
 the agnostic hypothesis test, $\phi_{T,c_0,c_1}$    
 (\cref{rule::agnostic}), where 
 $T=\left|\frac{\widehat{\beta}_k}{\sqrt{\widehat{\V} [\widehat{\beta}_k]}}\right|$.
 For every $\theta \in \Theta$, the probability that 
 $\phi_{T,c_0,c_1}$ accepts $H_0$ is 
 \begin{align}
  \P_\theta\left(T\leq c_0\right)
  &= \P_\theta\left(-c_0 \leq 
  \frac{\widehat{\beta}_k}
  {\sqrt{\widehat{\V}[\widehat{\beta}_k]}} \leq c_0\right)
  = \P_\theta\left(-c_0 \leq 
  \frac{\frac{\widehat{\beta}_k-\beta_k}
  {\sqrt{\V[\widehat{\beta}_k]}}+
  \frac{\beta_k}{\sqrt{\V[\widehat{\beta}_k]}}}
  {\sqrt{\widehat{\V}[\widehat{\beta}_k]
  {\V}[\widehat{\beta}_k]^{-1}}} \leq c_0\right) \notag \\
  &=\P\left(-c_0 \leq T_{n-d-1,\delta_k} \leq c_0\right),     
  \label{eq::cohen}
 \end{align}
 where $T_{p,\delta}$ has a
 non-central $t$-distribution with 
 $p$ degrees of freedom and 
 non-centrality parameters $\delta$, that is,
 $\delta_k=\frac{\beta_k}{\sqrt{\V[\widehat{\beta}_k]}}= \frac{d_k}{\sqrt{a}_k}$,
 $a_k$ is the $k$-th element of 
 the diagonal of the matrix $(\mX^t \mX)^{-1}$, and
 $d_k = \frac{\beta_k}{\sigma}$ is 
 the Cohen's $d$ effect size of 
 the $k$-th variable on $Y$
 \citep{Cohen1977}.

 A practitioner can determine 
 a desired Cohen's effect size value, $d_k^*$ and 
 a $\beta \in (0,1)$, and
 use \Cref{eq::cohen} to choose
 $c_0$ such that the type II error is $\beta$
 when the effect size is $d_k^*$.
 Since, when $\delta>\delta'$,
 $T_{p,\delta}$ stochastically dominates
 $T_{p,\delta'}$ this procedure guarantees that
 \begin{align*}
  \sup_{\theta \in \subH1}
  \P_\theta(\phi=0) = \beta,
 \end{align*}
 where $\subH1 = \{\theta \in \Theta:\delta_k\sqrt{a_k} \geq d_k^*\}$.
 That is, type II error probabilities
 are controlled by $\beta$ for 
 every parameter value with effect size
 greater or equal to $d_k^*$.
 Note that, when $d_k^*=0$, the
 test which is obtained is 
 the standard $(\alpha,\beta)$-level test
 for $H_0^=$ in \Cref{ex:regression}
 when $k=(0,\ldots,0,1,0,\ldots,0)$ and $c=0$.
\end{example}

The next example applies the derivation
in \Cref{ex:linearEffect} to a real dataset.

\begin{example}
 \label{ex::infant}
 The Swiss Fertility and 
 Socioeconomic Indicators (1888) Data \citep{Mosteller1977},
 contains a fertility measure and socio-economic indicators
 for 47 French-speaking provinces of Switzerland.
 \Cref{tab::swiss} presents the estimates of 
 regressing the infant mortality rate over 
 the other covariates using $d_k^*=0.25$,
 for every $k$, $\alpha=0.05$, and $\beta=0.2$.
 The analysis indicates that
 both the agriculture index of a province and
 the percentage of catholics on it 
 are not associated to its infant mortality rate. 
 On the other hand, there is 
 an association between 
 fertility and infant mortality rate. 
 Finally, it is not possible to assert whether 
 education and examination (percentage of 
 draftees receiving highest mark on army examination)
 are associated to the response variable.
 \Cref{img::power_reg} shows the probability of
 each decision as a function of 
 Cohen's $d$ effect size.

 \begin{table}[H]
  \centering
  \begin{tabular}{rlllll}
   \hline
   & Estimate & Std. Error & t-value & p-value & Decision \\ 
   \hline
  (Intercept) & 8.667 & 5.435 & 1.595 & 0.119 & Accept \\ 
    Fertility & 0.151 & 0.054 & 2.822 & 0.007 & Reject \\ 
  Agriculture & -0.012 & 0.028 & -0.418 & 0.678 & Accept \\ 
  Examination & 0.037 & 0.096 & 0.385 & 0.702 & Agnostic \\ 
    Education & 0.061 & 0.085 & 0.719 & 0.476 & Agnostic \\ 
     Catholic & 0.001 & 0.015 & 0.005 & 0.996 & Accept \\ 
  \hline
  \end{tabular}
  \caption{Agnostic regression analysis over 
  the Swiss dataset (\cref{ex::infant}).}
  \label{tab::swiss}
 \end{table}

 \begin{figure}[!htpb]
  \centering
  \includegraphics[width=0.6\textwidth]{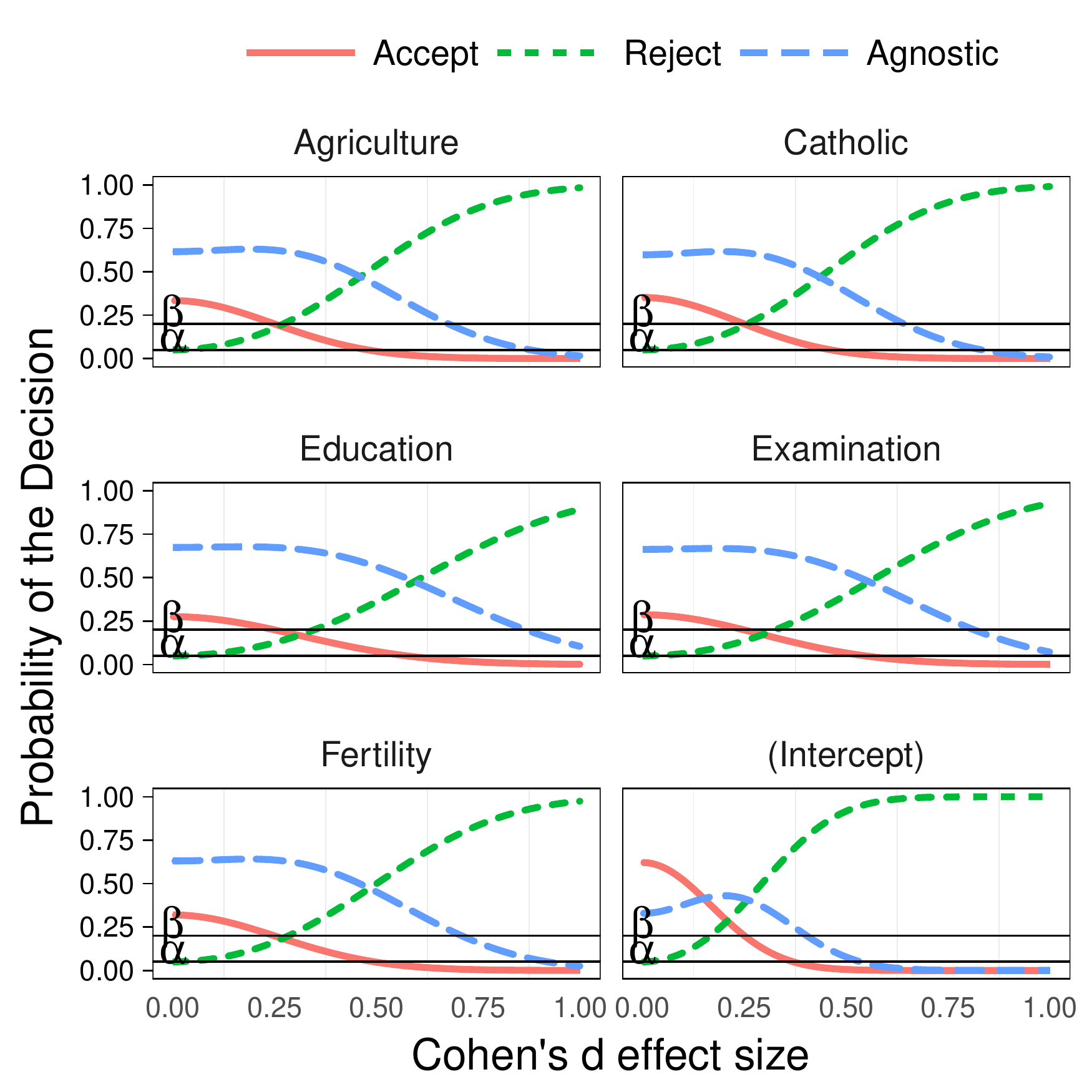}
  \caption{Probability of each decision for 
  the regression coefficients on 
  the Swiss dataset (\cref{ex::infant}).}
  \label{img::power_reg}
 \end{figure}
\end{example}

While controlling the type II error probabilities 
only for a class of effect sizes,
such as illustrated in \cref{ex:linearEffect,ex::infant},
it is possible to obtain a consistent 
sequence of agnostic hypothesis tests. 
\Cref{ex:consistency} illustrates this possibility
in a bilateral $z$-test.

\begin{example}
 \label{ex:consistency}
 Let $X_{1},\ldots,X_{n}$ be a i.i.d. sample with
 $X_{i} \sim N(\mu,\sigma^2)$, where
 $\mu \in \Re$ and $\sigma^2$ is known.
 Let $H_0: \mu=0$, $\alpha_n=\beta_n=\exp(-o(n))$,
 $a_n=\frac{-\Phi^{-1}(0.5\alpha_n)\sigma}{\sqrt{n}}$,
 $b_n$ be such that $b_n \leq a_n$ and
 $b_n^{-1}=o(\sqrt{n})$, 
 $\cc_n = (-a_n,-b_n,b_n,a_n)$, and
 $\gamma_n=b_n+\left(\frac{-2\log(\sqrt{2\pi}\beta_n)}{n}\right)^{0.5}$.
 The agnostic test $\phi_{\bar{X}_n,\cc_n}$
 controls the type I error by $\alpha_n$, and
 controls the type II error over 
 $H_1^*: |\mu| > \gamma_n$ by $\beta_n$.
 Furthermore, for every $\mu \in \Re$,
 $\lim_{n \rightarrow \infty}\pi_{\phi_{\bar{X}_n,\cc_n}}(\mu)=1$. That is,
 $\left(\phi_{\bar{X}_n,\cc_n}\right)_{n \in \mathbb{N}}$
 is consistent.
\end{example}

\Cref{ex:consistency} shows that, if
a sequence of tests doesn't control
the type II error probabilities in a
neighborhood of $H_0$, then it can be consistent.
This occurs because, 
contrary to $H_0$ and $H_1$ that satisfy
\cref{assumption:connected},
$H_0$ and $H_1^*$ are
``probabilistically separated''.
Also note that, since $\limn \gamma_n = 0$,
for every $\theta \in H_1$, 
there exists an $n^*$ after which
the error II probability for $\theta$ is
controlled by $\beta_n$.

\section{Final remarks}
\label{sec::final}

Since agnostic tests control the
type I and II error probabilities,
their outcomes are more interpretable than
the ones obtained using
standard hypothesis tests.
This paper provides several procedures to
construct agnostic tests.
In several statistical models, 
(unbiased) uniformly most powerful agnostic tests
are obtained.
When such tests are unavailable,
an alternative that is based on
standard p-values is presented.
The paper also provides
several links between region estimators 
and agnostic tests, which shows
in particular that
$(\alpha,\beta)$-level tests can be 
fully coherent from a logical perspective.
Finally, we have shown that 
although one cannot obtain consistency in
agnostic tests that control
type I and type II error probabilities uniformly,
this goal can be achieved by 
relaxing the control of the type II error probabilities.

An R package that implements several of 
the agnostic tests developed here is
available at
\url{https://github.com/vcoscrato/agnostic}.

\section{Acknowledgments}

This was partially funded  by Funda\c{c}\~ao de Amparo 
\`a Pesquisa do Estado de S\~ao Paulo (2017/03363-8).

\bibliography{agnostic-control}

\section{Demonstrations}
\label{sec:demo}

\begin{definition}
 Let $g_{0}(x) = \I(x=0)$ and
 $g_{1}(x) = \I(x=1)$.
\end{definition}

\begin{lemma}
 \label{lemma:standardize}
 For every agnostic test, $\phi$,
 \begin{enumerate}
  \item $g_{0}(\phi)$ and $g_{1}(\phi)$ 
  are standard tests.
  \item for every $\theta \in \Theta$,
  $\P_{\theta}(g_{0}(\phi)=1)=\P_{\theta}(\phi=0)$ and
  $\P_{\theta}(g_{1}(\phi)=1)=\P_{\theta}(\phi=1)$.
  \item If $\phi$ is unbiased, then
  $g_1(\phi)$ is unbiased for $H_0$ and
  $g_0(\phi)$ is unbiased for $H^*_0 = H_1$.
 \end{enumerate}
\end{lemma}

\begin{proof}[Proof of \cref{lemma:standardize}] \
 The first two items follow directly from
 \cref{def:std} and 
 the definitions of $g_0$ and $g_1$.
 Next, if $\phi$ is unbiased, then 
 $\alpha_{\phi}+\beta_{\phi} \leq 1$. Also,
 \begin{align*}
  \sup_{\theta_1 \in H_1}\P_{\theta_1}(g_{0}(\phi)=1)
  &= \sup_{\theta_1 \in H_1}\P_{\theta_1}(\phi=0) 
  = \beta_\phi \\
 \P_{\theta_0}(g_{0}(\phi)=1) 
 &= \P_{\theta_0}(\phi=0) 
 \geq \beta_\phi 
 & \text{for every } \theta_0 \in H_0
 \end{align*}
 That is, $g_{0}(\phi)$ is unbiased for $H^*_0$. 
 Similarly, $g_{1}(\phi)$ is unbiased for $H_0$.
\end{proof}

\begin{lemma}
 \label{lemma:k-r}
 Let $c, c_0, c_1 \in \Re$, $c_0 \leq c_1$ and 
 $\phi$ be an agnostic test.
 Also, define $H_0: \theta_0 \leq \theta^*$
 and $H_1: \theta > \theta^*$, and let
 $\theta_0 \in H_0$ and $\theta_1 \in H_1$. 
 Under \Cref{assumption:k-r}, 	
 \begin{enumerate}
  \item If $\P_{\theta^*}(\phi_{T,c_0,c_1}=1) \geq \P_{\theta^*}(\phi=1)$, then
  $\P_{\theta_1}(\phi_{T,c_0,c_1}=1) \geq \P_{\theta_1}(\phi=1).$
  \item If $\beta_{\phi_{T,c_0,c_1}} \geq \beta_{\phi}$, then
  $\P_{\theta_0}(\phi_{T,c_0,c_1}=0) \geq \P_{\theta_0}(\phi=0).$
  \item If $\P_{\theta^*}(\phi_{T,c,c}=1) = \P_{\theta^*}(\phi=1)$, then
  $\phi_{T,c,c} \succeq \phi$.
 \end{enumerate}
\end{lemma}

\begin{proof} \
 \begin{enumerate}
 
 \item Let $\theta_1 \in H_1$. Note that 
 $g_1(\phi_{T,c_0,c_1})=\phi_{T,c_1,c_1}$.
 Furthermore, it follows from 
 \Cref{lemma:standardize} that
 $\P_{\theta^*}(g_1(\phi_{T,c_0,c_1})=1) \geq \P_{\theta^*}(g_1(\phi)=1)$.
 Therefore, by defining $H_0^*:\theta=\theta^*$ 
 and $H_1^*:\theta=\theta_1$,
 it follows from 
 \Cref{assumption:k-r}.2 and
 the Neyman-Pearson lemma that
 $\P_{\theta_1}(g_{1}(\phi_{T,c_0,c_1})=1) \geq \P_{\theta_1}(g_{1}(\phi)=1)$. 
 The inequality
 $\P_{\theta_1}(\phi_{T,c_0,c_1}=1) \geq \P_{\theta_1}(\phi=1)$
 follows from \Cref{lemma:standardize}.
  
 \item Let $\theta_0 \in H_0$. Note that 
 $g_{0}(\phi_{T,c_0,c_1})=1-\phi_{T,c_0,c_0}$.
 Furthermore, it follows from 
 \Cref{lemma:standardize} that
 $\sup_{\theta_1 \in H_1}\P_{\theta_1}(g_0(\phi_{T,c_0,c_1})=1) \geq \sup_{\theta_1 \in H_1}\P_{\theta_1}(g_0(\phi)=1)$
 Therefore, by taking $H_0^*=H_1$ and
 $H_1^*=H_0$, it follows from
 \Cref{assumption:k-r}.2 and
 the Karlin-Rubin theorem that
 $\P_{\theta_0}(g_{1}(\phi_{T,c_0,c_1})=1) \geq \P_{\theta_0}(g_{1}(\phi)=1)$.
 It follows from \Cref{lemma:standardize} that
 $\P_{\theta_0}(\phi_{T,c_0,c_1}=0) \geq \P_{\theta_0}(\phi=0)$.
  
 \item It follows from \Cref{lemma:k-r}.1
 that, for every $\theta_1 \in H_1$,
 $\P_{\theta_1}(\phi_{T,c,c}=1) \geq \P_{\theta_1}(\phi=1)$. 
 Next, obtain from 
 $\P_{\theta^*}(\phi_{T,c,c}=1)=\P_{\theta^*}(\phi=1)$ 
 and $\phi_{T,c,c}$ being a standard test,
 that $\P_{\theta^*}(\phi_{T,c,c}=0) \geq \P_{\theta^*}(\phi=0)$. 
 It follows from \Cref{lemma:standardize} that
 $\P_{\theta^*}(g_{0}(\phi_{T,c,c})=1) \geq \P_{\theta^*}(g_{0}(\phi)=1)$.
 By taking $H_0:\theta=\theta^*$ and
 $H_1:\theta=\theta_0$, it follows from 
 \Cref{assumption:k-r}.2 and the 
 Neyman-Pearson lemma that
 $\P_{\theta_0}(g_{1}(\phi_{T,c,c})=1) \geq \P_{\theta_0}(g_{1}(\phi)=1)$.
 Obtain from \Cref{lemma:standardize} that
 $\P_{\theta_0}(\phi_{T,c,c}=0) \geq \P_{\theta_1}(\phi=0)$. 
 Conclude that
 $\phi_{T,c,c} \succeq \phi$.
 \end{enumerate}
\end{proof}

\begin{proof}[Proof of \Cref{thm:k-r}]
 Let $\phi$ be an arbitrary
 $(\alpha,\beta)$-size agnostic test.
 \begin{enumerate}
 \item Conclude from \Cref{assumption:k-r} that
 \begin{align}
  \label{eq:ump-1}
  P_{\theta^*}(\phi_{T,c_0,c_1}=1) &= \alpha
  \geq \alpha_{\phi}
  \geq \P_{\theta^*}(\phi=1) \nonumber \\
  \beta_{\phi_{T,c_0,c_1}} 
  &=\beta \geq \beta_{\phi}
 \end{align}
 It follows from \cref{eq:ump-1} and
 \Cref{lemma:k-r} that
 $\phi_{T,c_0,c_1} \succeq \phi$.
 Since $\phi$ was arbitrary, conclude that
 $\phi_{T,c_0,c_1}$ is an UMP
 $(\alpha,\beta)$-level agnostic test.
 
  \item Either there exists 
  $c \in [c_1,c_0]$ such that
 $\P_{\theta^*}(\phi_{T,c,c}=1)=\P_{\theta^*}(\phi=1)$ 
 or there exists no such $c$. 
 If there exists such a $c$, then 
 it follows from \Cref{lemma:k-r} that
 $\phi_{T,c,c} \succeq \phi$.
 Next, assume there exists no such $c$.
 Note that $\phi$ has size $(\alpha,\beta)$ 
 and, therefore,
 \begin{align*}
 \P_{\theta^*}(\phi_{T,c_0,c_0}=1)
 = \alpha \geq 
 \P_{\theta^*}(\phi=1).
 \end{align*}
 Since $\P_{\theta^*}(\phi_{T,c,c}=1)$ 
 decreases continuously over $c$, conclude that
 \begin{align}
  \label{eqn:ump-2}
  \P_{\theta^*}(\phi_{T,c_1,c_1}=1)
  &\geq \P_{\theta^*}(\phi=1) \nonumber \\
  \beta_{\phi_{T,c_1,c_1}} = \beta
  &\geq \beta_{\phi}
 \end{align}
 Conclude from \cref{eqn:ump-2} and
 \Cref{lemma:k-r} that
 $\phi_{T,c_1,c_1} \succeq \phi$.
 \end{enumerate}
\end{proof}

\begin{lemma}
 \label{lemma:umpu-1}
 Let $c, c_0, c_1 \in \Re$, $c_0 \leq c_1$ and 
 $\phi$ be an unbiased test.
 Define $H_0: \theta(1) \leq \theta^*$
 and $H_1: \theta(1) > \theta^*$ and let
 $\bar{\theta} \in \Theta$ be such that 
 $\theta(1)=\theta^*$. 
 Under \Cref{assumption:umpu-1}, 	
 \begin{enumerate}
  \item If $\P_{\bar{\theta}}(\phi_{V,c_0,c_1}=1) \geq \P_{\bar{\theta}}(\phi=1)$, then 
  $\forall \theta_1 \in H_1$,
  $\pi_{\phi_{V,c_0,c_1}}(\theta_1) \geq \pi_{\phi}(\theta_1)$.
  \item If $\P_{\bar{\theta}}(\phi_{V,c_0,c_1}=0) \geq \beta_{\phi}$, then $\forall \theta_0 \in H_0$,
  $\pi_{\phi_{V,c_0,c_1}}(\theta_0) \geq \pi_{\phi}(\theta_0)$.
 \end{enumerate}
\end{lemma}

\begin{proof} \
 \begin{enumerate}
  \item Let $\theta_1 \in H_1$. 
  We wish to show that 
  $\P_{\theta_1}(g_1(\phi_{V,c_0,c_1})=1) \geq \P_{\theta_1}(g_1(\phi)=1)$,
  Since $g_1(\phi_{V,c_0,c_1})$ and 
  $g_1(\phi)$ are standard tests,
  our strategy is to obtain the inequality from
  \citet{Lehmann2006}[p.151].
  In order to obtain this result,
  Assumption \ref{assumption:umpu-1} is used to 
  show that $g_1(\phi_{V,c_0,c_1})$ satisfies
  the required conditions.
  
  Let $\Theta^*=\{\theta \in \Theta: \theta(1) \geq \bar{\theta}\}$.
  Note that
  $g_{1}(\phi_{V,c_0,c_1})=\phi_{V,c_1,c_1}$. 
  Also, it follows from
  \Cref{lemma:standardize} that
  $g_1(\phi)$ is unbiased for $H_0$ under $\Theta$.
  Since $H_0$ is more restrictive under $\Theta^*$,
  $g_1(\phi)$ is also unbiased for 
  $H_0$ under $\Theta^*$.
  Moreover, it follows from
  \Cref{lemma:standardize} that
  $\P_{\bar{\theta}}(g_1(\phi_{V,c_0,c_1})=1) \geq \P_{\bar{\theta}}(g_1(\phi)=1)$.
  It follows from 
  \Cref{assumption:umpu-1} that,
  under $\Theta^*$,
  $\alpha_{g_1(\phi_{V,c_0,c_1})} \geq \alpha_{g_1(\phi)}$. Putting all of the 
  above conditions together, conclude that
  $\P_{\theta_1}(g_1(\phi_{V,c_0,c_1})=1) \geq \P_{\theta_1}(g_1(\phi)=1)$ by applying
  \citet{Lehmann2006}[p.151] in $\Theta^*$.
  It follows directly from  
  \Cref{lemma:standardize} that
  $\P_{\theta_1}(\phi_{V,c_0,c_1}=1) \geq \P_{\theta_1}(\phi=1)$,
  which is equivalent to, 
  $\pi_{\phi_{V,c_0,c_1}}(\theta_1) \geq \pi_{\phi}(\theta_1)$.
  
  \item Let $\theta_0 \in \{\theta \in \Theta: \theta(1) < \theta^*\}.$
  Note that
  $g_{0}(\phi_{V_c0,c_1})=1-\phi_{V,c_0,c_0}$. 
  Also, it follows from 
  \Cref{lemma:standardize} that
  $g_0(\phi)$ is unbiased for $H_0^*=H_1$.
  Also, obtain from 
  \Cref{lemma:standardize} and
  \Cref{assumption:umpu-1}.2 that
  $\P_{\bar{\theta}}(g_0(\phi_{V,c_0,c_1})=1) \geq \sup_{\theta_1 \in \Theta_1 \cup \{\theta\}}\P_{\theta_1}(g_0(\phi)=1)$.
  Therefore, by taking
  $H_0^*: \theta(1) \geq \theta^*$,
  it follows from
  from \Cref{assumption:umpu-1}
  and \citet{Lehmann2006}[p.151] that
  $\P_{\theta_0}(g_0(\phi_{V,c_0,c_1})=1) \geq \P_{\theta_0}(g_0(\phi)=1)$.
  Conclude from \Cref{lemma:standardize} that
  $\P_{\theta_0}(\phi_{V,c_0,c_1}=0) \geq 	 \P_{\theta_0}(\phi=0)$.
  Since $\theta_0$ was 
  arbitrary in $H_1^*$,
  conclude from 
  \Cref{assumption:umpu-1}.2 that,
  for every $\theta_0 \in \overline{H_1^*} = H_0$,
  $\P_{\theta_0}(\phi_{V,c_0,c_1}=0) \geq \P_{\theta_0}(\phi=0)$, that is,
  $\pi_{\phi_{V,c_0,c_1}}(\theta_0) \geq \pi_{\phi}(\theta_0)$.
  \end{enumerate}
\end{proof}

\begin{proof}[Proof of \Cref{thm:umpu-1}]
 Since $\alpha+\beta \leq 1$, obtain $c_0 \leq c_1$.
 It follows from \Cref{assumption:umpu-1} that
 $\phi_{V,c_0,c_1}$ is a $(\alpha,\beta)$-level test.
 Let $\phi$ be an unbiased $(\alpha,\beta)$-size test. 
 Therefore, note that
 $\P_{\bar{\theta}}(\phi_{V,c_0,c_1}=1) = \alpha \geq \alpha_{\phi}$ and 
 $\P_{\bar{\theta}}(\phi_{V,c_0,c_1}=0) = \beta \geq \beta_{\phi}$. 
 Conclude from \Cref{lemma:umpu-1} that
 $\phi_{V,c_0,c_1} \succeq \phi$.  
\end{proof}

\begin{proof}[Proof of \cref{thm:umpu-2}]
 Since $\alpha+\beta \leq 1$, obtain
 $c_{1,l} \leq c_{0,l} \leq c_{0,r} \leq c_{1,r}$. 
 Let $\phi$ be an unbiased 
 $(\alpha,\beta)$-size test and
 $\theta_1 \in H_1$.
 Since $\alpha_{g_1(\phi_{V,\cc})} = \alpha_{\phi_{V,\cc}}\geq \alpha_{\phi} = \alpha_{g_1(\phi)}$, it follows from
 \Cref{assumption:umpu-2} and
 \citet{Lehmann2006}[p.151] that one can obtain
 $\P_{\theta_1}(g_1(\phi_{V,\cc})=1) \geq \P_{\theta_1}(g_1(\phi)=1)$.
 Conclude from \Cref{lemma:standardize} that
 $\P_{\theta_1}(\phi_{V,\cc}=1) \geq \P_{\theta_1}(\phi=1)$, which is equivalent to,
 $\pi_{\phi_{V,\cc}}(\theta_1) \geq \pi_{\phi}(\theta_1)$.
 Next, let $\theta_0 \in H_0$.
 Since $\phi$ is an $(\alpha,\beta)$-size test,
 for every $\theta_1 \in H_1$,
 $\P_{\theta_1}(\phi=0) \leq \beta$.
 It follows from \Cref{assumption:umpu-2} that
 $\P_{\theta_0}(\phi=0) \leq \beta$, that is,
 $\pi_{\phi}(\theta_0) \leq \beta$.
 Since $\pi_{\phi_{V,\cc}}(\theta_0)=\beta$,
 obtain $\pi_{\phi_{V,\cc}}(\theta_0) \geq \pi_{\phi}(\theta_0)$.
\end{proof}




%
%

\begin{definition}
 \label{defn:unbiased-stat}
 A statistic, $T \in \Re$, 
 is unbiased for $H_0$ if, 
 for every $t \in \mathbb{R}$,
 $\theta_0 \in H_0$ and
 $\theta_1 \in H_1$,
 $\P_{\theta_0}(T \leq t) \geq \P_{\theta_1}(T \leq t)$.
\end{definition}

\begin{assumption} \
 \label{assumption:p}
 \begin{enumerate}
  \item $\Theta$ is 
  a connected space.
  \item $T$ is 
  an unbiased statistic for $H_0$.
  \item For every $t \in \Re$,
  $\P_{\theta}(T \geq t)$
  is a continuous function over $\theta$.
 \end{enumerate}
\end{assumption}

\begin{lemma}
 \label{lemma:p-1}
 Under \cref{assumption:p},
 for every $t \in \Re$,
 \begin{align*}
  \sup_{\theta_0 \in H_0}\P_{\theta_0}(T > t)
  &= 1-\sup_{\theta_1 \in H_1}
  \P_{\theta_1}(T \leq t) 
 \end{align*}
\end{lemma}

\begin{proof} 
 Let $\partial H_0$ and $\partial H_1$ denote 
 the boundaries of $H_0$ and $H_1$. Since
 $H_1 = H_0^c$, $\partial H_0 = \partial H_1$.
 Also, since $\Theta$ is connected, 
 $\partial H_0 \neq \emptyset$. Therefore,
 \begin{align}
  \label{eq:pval-1}
  \sup_{\theta_0 \in H_0}\P_{\theta_0}(T > t)
  &= 1-\inf_{\theta_0 \in H_0}
  \P_{\theta_1}(T \leq t) 
  \nonumber \\
  &\geq 1-\inf_{\theta_0 \in \partial H_0}
  \P_{\theta_0}(T \leq t)
  & \text{\cref{assumption:p}.3} 
  \nonumber \\
  &= 1-\inf_{\theta_1 \in \partial H_1}
  \P_{\theta_1}(T \leq t) \nonumber \\
  &\geq 1-\sup_{\theta_1 \in H_1}
  \P_{\theta_1}(T \leq t)
  & \text{\cref{assumption:p}.3}
 \end{align}
 Furthermore,
 \begin{align}
  \label{eq:pval-2}
  \sup_{\theta_0 \in H_0} \P_{\theta_0}(T > t)
  &= 1-\inf_{\theta_0 \in H_0}
  \P_{\theta_0}(T \leq t) 
  \nonumber \\
  &\leq 1-\sup_{\theta_1 \in H_1}
  \P_{\theta_1}(T \leq t) 
  & \text{\cref{assumption:p}.2}
 \end{align}
 The proof follows from
 \cref{eq:pval-1,eq:pval-2}.
\end{proof}

\begin{lemma}
 \label{lemma:unbiased-p}
 If $\Phi$ is a nested family of 
 standard tests for $H_0$ such that, 
 for every $\phi \in \Phi$,
 $\phi$ is unbiased for $H_0$, then
 $1-p_{H_0,\Phi}$ is 
 an unbiased statistic for $H_0$.
\end{lemma}

\begin{proof}
 For each $t \in [0,1]$, let
 $\phi^*_t \in \Phi$ be 
 such that $\alpha_{\phi^*_t}=t$.
 \begin{align*}
  \P_{\theta_0}(1-p_{H_0} \leq t)
  &= 1-\P_{\theta_0}(p_{H_0} < 1-t) \\
  &= 1-\P_{\theta_0}(\phi^*_{1-t} = 1) \\
  &\geq 1-\alpha_{\phi^*_{1-t}} \\
  &\geq 1-\P_{\theta_1}(\phi^*_{1-t}=1) \\
  &= 1-\P_{\theta_1}(p_{H_0} < 1-t)
  = \P_{\theta_1}(1-p_{H_0} \leq t)
 \end{align*}
\end{proof}

\begin{proof}[Proof of \cref{thm:p}]
 \begin{align*}
  \alpha_{\phi_{\alpha,\beta}}
  &= \sup_{\theta_0 \in H_0} 
  \P_{\theta_0}(1-p_{H_0}(X) > 1-\alpha) \\
  &= \sup_{\theta_0 \in H_0}
  \P_{\theta_0}(p_{H_0}(X) < \alpha)
  = \alpha \\
  \beta_{\phi_{\alpha,\beta}}
  &= \sup_{\theta_1 \in H_1} 
  \P_{\theta_1}(1-p_{H_0}(X) \leq \beta) \\
  &= \sup_{\theta_1 \in H_1}
  \P_{\theta_1}(p_{H_0}(X) \geq 1-\beta) \\
  &= \sup_{\theta_0 \in H_0}
  1-\P_{\theta_0}(p_{H_0}(X) < 1-\beta)
  = \beta 
  & \text{\cref{lemma:p-1,lemma:unbiased-p}}
 \end{align*}
\end{proof}

\begin{proof}[Proof of \cref{thm:region-control}]
 Since $R(x)$ has confidence $1-\alpha$,
 $\P_{\theta}(\theta \notin R(x)) \geq \alpha$,
 for every $\theta \in \Theta$. Therefore,
 \begin{align*}
  \alpha_{\phi_{R,H_0}}
  &= \sup_{\theta_0 \in H_0}
  {\P_{\theta_0}(\phi_{R,H_0}=1)} 
  =  \sup_{\theta_0 \in H_0}
  {\P_{\theta_0}(R(X) \subseteq H_0^c)}
  \leq \sup_{\theta_0 \in H_0}
  {\P_{\theta_0}(\theta_0 \notin R(X))}
  \leq \alpha \\
  \beta_{\phi_{R,H_0}}
  &= \sup_{\theta_1 \in H_1}
  {\P_{\theta_1}(\phi_{R,H_0}=0)} 
  =  \sup_{\theta_1 \in H_1}
  {\P_{\theta_1}(R(X) \subseteq H_0)}
  \leq \sup_{\theta_1 \in H_1}
  {\P_{\theta_1}(\theta_1 \notin R(X))}
  \leq \alpha
 \end{align*}
\end{proof}

\begin{proof}[Proof of \cref{thm:region-unilateral}]
 Let $\theta^* \in \Re$ and
 $R_1(x)$ be a set. We write $\theta^* < R_1(x)$ if,
 for every $\theta(1) \in R_1(x)$,
 $\theta^* < \theta(1)$. 
 Also, $\theta^* > R_1(x)$ if,
 for every $\theta(1) \in R_1(x)$,
 $\theta^* > \theta(1)$.

 For each $x \in \sX$, let
 $R_1(x)=\left\{\theta(1): \phi_{H_{0,\theta(1)}}(x) = \half \right\}$. 
 If $\phi_{H_{0,\theta^*}}(x)=1$, then
 conclude from \cref{assumption:region-unilateral-1}
 that for every $\theta(1) \leq \theta^*$,
 $\phi_{H_{0,\theta(1)}}(x)=1$.
 Therefore, if $\phi_{H_{0,\theta^*}}(x)=1$,
 $\theta^* < R_1(x)$.
 Similarly, if $\phi_{H_{0,\theta^*}}(x)=0$,
 then it follows from 
 \cref{assumption:region-unilateral-1} that
 $\theta^* > R_1(x)$.
 Since $\phi_{H_{0,\theta^*}}(x) \in \left\{0,\half,1\right\}$, conclude that
 $\phi_{H_{0,\theta^*}}(x)=1$ if and only if
 $\theta^* < R_1(x)$  and
 $\phi_{H_{0,\theta^*}}(x)=0$ if and only if
 $\theta^* > R_1(x)$.
 That is, for every $\theta^*$,
 $\phi_{H_{0,\theta^*}}$ is 
 based on $R(x):=R_1(x)\times \mathbb{R} \times \ldots \times \mathbb{R}$ for $H_{0,\theta^*}$.
 
 Finally, if
 \cref{assumption:region-unilateral-2} holds, then
 for every $\theta \in \Theta$,
 \begin{align*}
\P_{\theta}(\theta \in R(X))=  \P_{\theta}(\theta(1) \in R_1(X))
  &= \P_{\theta}\left(\phi_{H_{0,\theta(1)}}
  =\half\right) \geq 1-2\alpha
 \end{align*}
 That is, $R_1(X)$ has 
 confidence $1-2\alpha$ for $\theta(1)$
 and $R(X)$ has 
 confidence $1-2\alpha$ for $\theta$.
\end{proof}

\begin{proof}[Proof of \cref{cor-izbicki}]
 Follows directly from
 \cref{thm:region-logical,thm:region-unilateral,thm:region-control}.
\end{proof}

\begin{proof}[Proof of \cref{cor:ump-region}]
 Let $T$ be such as in \cref{assumption:k-r} and
 $\theta_1, \theta_2, \theta_3 \in \Theta$ be
 such that
 $\theta_1 \leq \theta_2 \leq \theta_3$.
 It follows from \cref{thm:k-r} that
 $\phi_{H_{0,\theta_i}} = \phi_{T,c_{0,\theta_i},c_{0,\theta_i}}$,
 where $c_{0,\theta_i}$ and $c_{1,\theta_i}$
 are such that
 $\sup_{\theta_1 \in H_{1,\theta_i}}\P_{\theta_1}(T \leq c_{0,\theta_i})=\alpha$ and
 $\sup_{\theta_0 \in H_{0,\theta_i}}\P_{\theta_0}(T > c_{1,\theta_i})=\alpha$.
 Since $\theta_1 \leq \theta_2$,
 $H_{0,\theta_1} \subset H_{0,\theta_2}$.
 Therefore, $c_{1,\theta_1} \leq c_{1,\theta_2}$,
 that is, if 
 $\phi_{T,c_{0,\theta_2},c_{1,\theta_2}}(x)=1$, then
 $\phi_{T,c_{0,\theta_1},c_{1,\theta_1}}(x)=1$.
 Similarly, if
 $\phi_{T,c_{0,\theta_3},c_{1,\theta_3}}(x)=0$, then
 $\phi_{T,c_{0,\theta_2},c_{1,\theta_2}}(x)=0$.
 Conclude that, if
 $\phi_{H_{0,\theta_2}}(x)=0$, then
 $\phi_{H_{0,\theta_3}}(x)=0$ and, if
 $\phi_{H_{0,\theta_2}}(x)=1$ then
 $\phi_{H_{0,\theta_1}}(x)=1$.
 Also, for every $\theta^* \in \Theta$,
 it follows from \cref{thm:k-r} and
 the continuity of $\P_{\theta}(T \leq t)$
 over $\theta$ that
 $\P\left(\phi_{H_{0,\theta^*}}=\half\right)=1-2\alpha$. The proof follows directly from
 \cref{thm:region-unilateral}.
\end{proof}

\begin{proof}[Proof of \cref{cor:umpu-region}]
 Since $g(V_{\theta},\theta)$ is ancillary,
 there exist $v_{\alpha}$ and $v_{1-\alpha}$
 such that, for every $\theta \in \Theta$,
 $\P_{\theta}((V_{\theta},\theta) \leq v_{\alpha}) = \alpha$ and 
 $\P_{\theta}(g(V_{\theta},\theta) > v_{1-\alpha}) = \alpha$. Since $g(v,\theta)$ is 
 decreasing over $\theta$,
 for every $\theta \in \Theta$,
 $\P_{\theta}(V_{\theta} \leq g^{-1}(v_{\alpha},\theta)) = \alpha$ and 
 $\P_{\theta}(V_{\theta} > g^{-1}(v_{1-\alpha},\theta)) = \alpha$. Conclude from
 \cref{thm:umpu-1} that
 \begin{align}
  \label{eq:umpu-region-1}
  \phi_{H_{0,\theta^*}}
  =\phi_{V_{\theta^*},
  g^{-1}(v_{\alpha},\theta^*),
  g^{-1}(v_{1-\alpha},\theta^*)}
 \end{align}
 Let $\theta_1 \leq \theta_2 \leq \theta_3$.
 Since $g^{-1}(v,\theta)$ is
 increasing over $\theta$, conclude from
 \cref{eq:umpu-region-1} that,
 if $\phi_{H_{0,\theta_2}}(x)=1$, then
 $\phi_{H_{0,\theta_1}}(x)=1$. Also,
 if $\phi_{H_{0,\theta_2}}(x)=0$, then
 $\phi_{H_{0,\theta_3}}(x)=0$.
 Also, it follows from \cref{thm:umpu-1} that,
 for every $\theta \in \Theta$ such that
 $\theta(1)=\theta^*$,
 $\P_{\theta}\left(\phi_{H_{0,\theta^*}}=\half\right)=1-2\alpha$.
 The proof follows directly from 
 \cref{thm:region-unilateral}.
\end{proof}

\begin{proof}[Proof of \cref{thm:region-bilateral}]
 Let 
 \begin{align*}
  R^{(1)}_1(x) &=
  \left\{\theta^* \in \Re: 
  \phi_{H_{0,\theta^*}} = 0\right\} \\
  R^{(1)}_{2}(x) &=
  \left\{\theta^* \in \Re: 
  \phi_{H_{0,\theta^*}} \in 
  \left\{0,\half\right\}\right\} \\
  R_1(x)&=R^{(1)}_1(x)\times \mathbb{R} \times \ldots \times \mathbb{R} \\
  R_2(x)&=R^{(1)}_2(x)\times \mathbb{R} \times \ldots \times \mathbb{R}
 \end{align*}
 By construction $R_1(x) \subseteq R_2(x)$,
 $\phi_{H_{0,\theta^*}}(x)=0$ if and only if
 $\{\theta^*\} \subseteq R^{(1)}_1(x)$ (and thus
 $\phi_{H_{0,\theta^*}}(x)=0$ if and only if
 $H_{0,\theta^*} \subseteq R_1(x)$)
 and
 $\phi_{H_{0,\theta^*}}(x)=1$ if and only if
 $R^{(1)}_2(x) \subseteq \{\theta^*\}^c$
 (and thus
 $\phi_{H_{0,\theta^*}}(x)=1$ if and only if
 $R_2(x) \subseteq H^c_{0,\theta^*} $).
 That is, $\phi_{H_{0,\theta^*}}(x)$ is
 based on $R_1(x)$ and $R_2(x)$.
 Furthermore, for every $\theta \in \Theta$,
 \begin{align*}
\P_{\theta}(\theta \in R_1(X))=  \P_{\theta}(\theta(1) \in R^{(1)}_1(X)) &=
  \P_{\theta}(\phi_{H_{0,\theta(1)}}=0)
  = \beta \\
\P_{\theta}(\theta  \notin R_2(X))=  \P_{\theta}(\theta(1) \notin R^{(1)}_2(X)) &=
  \P_{\theta}(\phi_{H_{0,\theta(1)}}=1)
  = \alpha
 \end{align*}
 Conclude that $R_1(X)$ and $R_2(X)$ are
 confidence regions with confidence of,
 respectively, $\beta$ and $1-\alpha$.
\end{proof}

\begin{proof}[Proof of \cref{thm:no-consistency}]
 Since $\Theta$ is connected and
 $H_0 \notin \{\emptyset,\Theta\}$,
 $\partial H_0 \neq \emptyset$. Let
 $\theta^* \in \partial H_0$.
 If $\phi_n$ has size $(\alpha_n, \beta_n)$,
 $\sup_{\theta_0 \in H_0}\P_{\theta_0}(\phi_n=1) \leq \alpha_n$ 
 and $\sup_{\theta_1 \in H_1}\P_{\theta_1}(\phi_n=0)\leq \beta_n$.
 It follows from the continuity of
 $\P_{\theta}(\phi_n=i)$ that
 $\P_{\theta^*}(\phi_n=1) \leq \alpha_n$ and
 $\P_{\theta^*}(\phi_n=0) \leq \beta_n$.
 Therefore, for the first part of the theorem, 
 $\pi_{\phi_n}(\theta^*) \leq \max(\alpha, \beta) < 1$.
 That is, $\limn \pi_{\phi_n}(\theta^*) \neq 1$ and
 $\seqn{\phi}$ is not consistent.
 For the second part of the theorem,
 Since $\limn \alpha_n = \limn \beta_n = 0$,
 one obtains that
 $\limn \P_{\theta^*}\left(\phi_n = \half\right) = 1$.
\end{proof}

\end{document}